\pdfoutput=1
\documentclass[mathematics,article,accept,pdftex,oneauthor]{Definitions/mdpi} 
\firstpage{1} 
\pubvolume{1}
\issuenum{1}
\articlenumber{0}
\pubyear{2025}
\copyrightyear{2025}
\externaleditor{Adara M. Blaga} 
\datereceived{1 July 2025} 
\daterevised{29 July 2025} 
\dateaccepted{1 August 2025} 
\datepublished{ } 
\hreflink{https://doi.org/} 

\Title{Advanced Manifold--Metric Pairs} 

\TitleCitation{Advanced Manifold--Metric Pairs}

\Author{Pierros 
Ntelis}

        \AuthorCitation{Ntelis, P.}

\address [1]{%
Independent Researcher, 
13 Allee Turcat Mery, 13008 Marseille, France; ntelis.pierros@gmail.com}

\abstract{
This article presents a novel mathematical formalism for advanced manifold--metric pairs, enhancing the frameworks of geometry and topology. We construct various D-dimensional manifolds and their associated metric spaces using functional methods, with a focus on integrating concepts from mathematical physics, field theory, topology, algebra, probability, and statistics. Our methodology employs rigorous mathematical construction proofs and logical foundations to develop generalized manifold--metric pairs, including homogeneous and isotropic expanding manifolds, as well as probabilistic and entropic variants. Key results include the establishment of metrizability for topological manifolds via the Urysohn Metrization Theorem, the formulation of higher-rank tensor metrics, and the exploration of complex and quaternionic codomains with applications to cosmological models like the expanding spacetime. By combining spacetime generalized sets with information-theoretic and probabilistic approaches, we achieve a unified framework that advances the understanding of manifold--metric interactions and their physical implications.
}



\begin{document}
%
%
\tableofcontents

\section{Introduction}
\label{Introduction}

{Currently
, natural systems, such as those in physics and cosmology, are often modelled using algebraic and topological frameworks, including subfields like differential geometry. Recognizing that a space is a mathematical structure of arbitrary finite dimension, where objects have relative positions, we seek to explore its abstract properties, such as generalized metrics or probabilistic structures. In~differential geometry, a~manifold \(\mathcal{M}\) is a topological space that is locally Euclidean and equipped with properties like smoothness, which we generalize by associating with non-traditional metrics, $m$, such as higher-rank tensors}.

A manifold is essentially a space that resembles locally a Euclidean space in the sense that $\mathcal{M}$ is covered by some arbitrary coordinate patches. These kinds of structures allow differentiation to be defined, but~they do not distinguish intrinsically between different coordinate systems. In~turn, the~only concepts defined by the manifold structure are those that are independent of the choice of a coordinate system. This framework is rooted in differential geometry, algebraic geometry, and topology, where manifolds are equipped with a differentiable structure, enabling the definition of tangent spaces, tensor fields, and~metrics~\cite{Lee:2012, Wald:1984}. Furthermore, analysis on manifolds provides tools to study smooth functions and their properties, such as differentiability and integrability, which are essential for our generalization~\cite{Munkres:1991, Spivak:1965}. To~build such {novel} objects, we need tensors, and/or advanced generalized tensors~\cite{sym17050777}. {Advanced generalized tensors are novel generalized tensors that are defined using not only objects from standard calculus, but~also objects from fractional calculus~\cite{axioms11020043}.} 
In this manuscript, we will focus on the tensors that have already been explored. Note that manifolds can also be viewed from the perspective of category theory~\cite{sym17050777}. A~formal formulation can be found in \citet{hawking_ellis_1973}. However, in~this study, the former formulation is~sufficient. 

The motivation of our studies lies in the fact that, as a mathematics and physics community, we are interested in building higher and more sophisticated mathematical and theoretical physics structures {than the ones we currently have,} with the goal of understanding our universe at a deeper level.{ Therefore, this study enriches recent developments of manifold--metric pair constructions~\cite{math13060990}.}

Furthermore, a~plethora of metrics exists, which we can construct beyond the standard ones, starting from the simplest ones classified by \citet{bianchi1897sugli}, known as the Bianchi classification. These metrics are typically defined on Riemannian or pseudo-Riemannian manifolds, where properties such as symmetry and positive-definiteness (for Riemannian metrics) are well-established~\cite{ONeill:1983}. Interestingly, several studies are currently exploring the possibility that one of these new metrics possesses the property of extra dimensions. Notably,\linebreak  \citet{polyakov1981quantum,Deser:1976eh} have considered theories that include the ones discussed in a string theory framework. An~easy way to express this argument is a set of different types of dimensions; there is the extra dimension set, as well as a spectrum of dimensions, which can be represented in~a continuous manner, such as a probabilistic set of~dimensions.  

{In this context, the~\textit{{manifold dimension}} 
refers to the number of independent coordinates defining \(\mathcal{M}\), typically a fixed integer \(D\) \cite{Lee2018}. \textit{{Extra dimensions}} extend this concept to higher-dimensional manifolds, such as the 11-dimensional manifolds considered in cosmological models. \textit{{Spectrum dimensions}} arise from spectral properties of operators on \(\mathcal{M}\), such as eigenvalues of the metric tensor, which may characterize geometric or physical properties. \textit{Probabilistic sets of dimension} incorporate stochastic structures, where dimensions are defined via probabilistic measures or entropic quantities, as~explored in our generalized metric framework with higher-rank tensors, such as the $(0,3)$-tensor. These concepts are interconnected, as~extra dimensions expand the manifold’s structure, spectrum dimensions provide analytical tools, and~probabilistic dimensions introduce stochastic generalizations, collectively enabling novel manifold--metric pairs for applications in cosmology \linebreak  and beyond.} 

{We designate a functional tensor as a generalized metric because it extends the role of standard metrics in measuring geometric properties on the manifold $\mathcal{M}$, enabling the modeling of higher-order interactions, such as those in multiple geometric spaces, and~topological spaces, probabilistic or entropic spacetimes, which are critical for applications in cosmology and theoretical physics. Like standard pseudo-Riemannian metrics, such tensors are symmetric and define geometric structures, but~their trilinear or~multilinear forms allow for richer interactions beyond the bilinear framework of traditional metrics~\cite{Lee2018, ONeill1983, Kunzinger2001}.}

In fact, even probabilistic concepts in gravitational theories have not been thoroughly studied in science, yet there is significant interest in the subject~\cite{cho1996fixed,2017arXiv171206127K,article_Bailleul_Ismael,2018arXiv180100581B}. In~1996,\linebreak   \citet{cho1996fixed} studied several properties of probabilistic metric spaces, while in 2018,\linebreak   \citet{2018arXiv180100581B} introduced and studied the concepts of probabilistic 1-Lipschitz maps. The~latter authors used the space of all probabilistic 1-Lipschitz maps to present a new method for constructing a probabilistic metric completion. In~particular, they expanded the concepts of probabilistic invariant metric group completion. Proving that the space of all probabilistic 1-Lipschitz maps is defined on a probabilistic invariant metric group, then the latter group can be endowed with a semigroup structure. This probabilistic invariant completes Menger groups, which were characterized by the space of all probabilistic 1-Lipschitz maps and which map functionals in the spirit of the classical Banach--Stone theorem~\cite{2018arXiv180100581B}. Fast forward in time, the~mathematical structure of manifold--metric pairs is still studied under different Riemannian and non-Riemannian geometries~\cite{CANTATA:2021ktz,Capozziello:2022zzh}.

It is worth noting that we have some mild, yet important, evidence for probabilistic dimensions, as~we can observe by the limits on the number of dimensions $D$ of the spacetime continuum, which has been measured to be $D\simeq 4\pm0.1$ from gravitational wave estimates~\cite{2018JCAP...07..048P}. These results suggest not only that our standard paradigm, which is described with {$(1,3)$}
-dimensional manifolds, is valid, but~also that there is further room for improvements and generalization of these concepts. These findings result in an inaugural interest in concepts of probabilistic spaces in the literature, which motivates us to explore further these probabilistic concepts and possibly apply them to $D$-spaces and $D$-dimensional manifolds. 
Note that these theories will be tested using current and future cosmological experiments with novel methods such as those found by \citet{2021arXiv210610278P}. \linebreak  These advanced manifold--metric pairs can potentially be studied in the framework of functors of action theories~\cite{NtelisMorris_FA,Ntelis2024} and in the framework of analytical dynamical analyses, as carried out in previous studies~\cite{Ntelis_Said:2025a,Ntelis_Said:2025b,Ntelis_Said:2025d}. Note that as in the \citet{sym17050777} companion paper, we present some applications. In~this study, we go beyond the concepts of generalization of functional metrics studied by \citet{2018arXiv180300240J}.

\section{Advanced Manifold--Metric Space~Pairs}

To interpret experimental data across scales, from~astronomical to cosmological, physical systems are typically described using differential equations on {(1,3)-dimensional manifolds (i.e., one temporal and three spatial dimensions)} with corresponding metrics. These manifolds, often equipped with pseudo-Riemannian metrics of the Lorentzian signature, model phenomena in general relativity and cosmology. In~this study, we generalize the concepts of manifold and metric spaces to accommodate higher-rank tensors, functional dependencies, and~extended codomains, enabling broader applications in theoretical physics. We formalize this generalization through a theorem and supporting propositions, systematically constructing a generalized manifold--metric pair from the standard~framework.

{
\subsection*{Generalized 
Manifold--Metric Pair~Theorem}
\label{sec:Advancedmanifold-metric_pairs}

\begin{Theorem}[Generalized Manifold--Metric Pair]
Let $\mathcal{M}^D$ be a $D$-dimensional differentiable manifold. There exists a generalized manifold--metric pair $(\mathcal{M}_G, \mathcal{F})$, where $\mathcal{M}_G$ is a manifold supporting a $(U,L)$-rank functional tensor $\mathcal{F}_{\mu_1 \dots \mu_L}^{\nu_1 \dots \nu_U}[c, f(c)] : C_{\mu_1} \otimes \dots \otimes C_{\nu_U} \to \mathbb{Q}$, with~$C_{\alpha_i}$ being $D$-dimensional coordinate spaces, $c \in C$, $f(c) : C \to \mathbb{R}$ a smooth function, and~$\mathbb{Q}$ the quaternion field. The~corresponding line element is as follows:
\begin{equation}
    ds^{U+L} = \mathcal{F}_{\mu_1 \dots \mu_L}^{\nu_1 \dots \nu_U}[c, f(c)] dc^{\mu_1} \dots dc^{\mu_L} dc_{\nu_1} \dots dc_{\nu_U},
\end{equation}
where $dc_{\nu_i} = g_{\nu_i \lambda} dc^\lambda$ for a symmetric, positive-definite (0,2)-rank tensor $g_{\mu\nu}$. The~tensor $\mathcal{F}$ transforms as a $(U,L)$-tensor under coordinate changes and may be symmetric in its covariant and contravariant indices separately, ensuring a well-defined distance function.
\end{Theorem}

\begin{proof}
We construct 
the generalized manifold--metric pair $(\mathcal{M}_G, \mathcal{F})$ from the standard pair $(\mathcal{M}, m)$ through a series of propositions, each addressing a key step: extending the tensor rank, functionalizing the tensor, and~generalizing the~codomain.

\begin{Proposition}[Extension to (U,L)-Rank Tensor]
A $D$-dimensional differentiable manifold $\mathcal{M}$ with a standard manifold--metric pair $(\mathcal{M}, m_{\mu\nu} : C \otimes C \to \mathbb{R})$, where $m_{\mu\nu}$ is a symmetric (0,2)-tensor, can be generalized to a {proposition} pair $(\mathcal{M}, m_{\mu_1 \dots \mu_L}^{\nu_1 \dots \nu_U} : C_{\mu_1} \otimes \dots \otimes C_{\nu_U} \to \mathbb{R})$, where $m_{\mu_1 \dots \mu_L}^{\nu_1 \dots \nu_U}$ is a $(U,L)$-rank tensor.
\end{Proposition}

\begin{proof}
Consider a standard manifold--metric pair:
\begin{equation}
    (\mathcal{M}, m) \equiv (\mathcal{M}, m_{\mu\nu} : C \otimes C \rightarrow \mathbb{R}),
\end{equation}
where $\mathcal{M}$ is a $D$-dimensional differentiable manifold, $C$ is the $D$-dimensional coordinate space with coordinates $c^\mu$, and~$m_{\mu\nu}$ is a symmetric (0,2)-tensor, defining the line element as follows:
\begin{equation}
    ds^2 = m_{\mu\nu}(c) dc^\mu dc^\nu.
\end{equation}

{The} 
 metric $m_{\mu\nu}$ is symmetric ($m_{\mu\nu} = m_{\nu\mu}$) and transforms under coordinate changes $c^\mu \to c^{\mu'} = f^{\mu'}(c)$ as follows:
\begin{equation}
    m_{\mu'\nu'}(c') = \frac{\partial c^\rho}{\partial c^{\mu'}} \frac{\partial c^\sigma}{\partial c^{\nu'}} m_{\rho\sigma}(c),
\end{equation}
ensuring invariance of $ds^2$ \cite{Wald:1984}.

Define coordinate spaces $C_{\alpha_i}$, each a copy of $C$, for~$\alpha_i \in \{\mu_1, \dots, \mu_L, \nu_1, \dots, \nu_U\}$. Construct the tensor product spaces as follows:
\begin{align}
    C_L &:= C_{\mu_1} \otimes C_{\mu_2} \otimes \dots \otimes C_{\mu_L}, \\
    C_U &:= C_{\nu_1} \otimes C_{\nu_2} \otimes \dots \otimes C_{\nu_U},
\end{align}
forming $C_L \otimes C_U = C_{\mu_1} \otimes \dots \otimes C_{\nu_U}$. Define a $(U,L)$-rank tensor $m_{\mu_1 \dots \mu_L}^{\nu_1 \dots \nu_U} : C_L \otimes C_U \to \mathbb{R}$, which transforms as follows:
\begin{equation}
    m_{\mu_1' \dots \mu_L'}^{\nu_1' \dots \nu_U'} = \left( \prod_{i=1}^U \frac{\partial c^{\nu_i'}}{\partial c^{\rho_i}} \right) \left( \prod_{j=1}^L \frac{\partial c^{\sigma_j}}{\partial c^{\mu_j'}} \right) m_{\sigma_1 \dots \sigma_L}^{\rho_1 \dots \rho_U}.
\end{equation}

{This} tensor preserves the differentiable structure of $\mathcal{M}$ \cite{Lee:2012}. The~manifold-metric \linebreak  pair becomes the following:
\begin{equation}
    (\mathcal{M}, m) \equiv \left( \mathcal{M}, m_{\mu_1 \dots \mu_L}^{\nu_1 \dots \nu_U} : C_{\mu_1} \otimes \dots \otimes C_{\nu_U} \rightarrow \mathbb{R} \right).
\end{equation}

{Symmetry} in indices (e.g., $m_{\mu_1 \dots \mu_L}^{\nu_1 \dots \nu_U} = m_{\mu_{\pi(1)} \dots \mu_{\pi(L)}}^{\nu_1 \dots \nu_U}$ for permutations $\pi$) may be imposed \linebreak  as needed.
\end{proof}

\begin{Proposition}[Functional Tensor Generalization]
The $(U,L)$-rank tensor $m_{\mu_1 \dots \mu_L}^{\nu_1 \dots \nu_U}$ can be promoted to a functional tensor $\mathcal{F}_{\mu_1 \dots \mu_L}^{\nu_1 \dots \nu_U}[c, f(c)]$, where $f(c) : C \to \mathbb{R}$ is a smooth function, mapping $C_{\mu_1} \otimes \dots \otimes C_{\nu_U} \to \mathbb{R}$.
\end{Proposition}

\begin{proof}
Start with the pair $(\mathcal{M}, m_{\mu_1 \dots \mu_L}^{\nu_1 \dots \nu_U} : C_{\mu_1} \otimes \dots \otimes C_{\nu_U} \to \mathbb{R})$. Introduce a smooth function $f(c) : C \to \mathbb{R}$, where $c \in C$. Define the functional tensor $\mathcal{F}_{\mu_1 \dots \mu_L}^{\nu_1 \dots \nu_U}[c, f(c)]$, which depends on both coordinates $c$ and the function $f(c)$. The~tensor $\mathcal{F}$ inherits the transformation properties of $m_{\mu_1 \dots \mu_L}^{\nu_1 \dots \nu_U}$:
\begin{equation}
    \mathcal{F}_{\mu_1' \dots \mu_L'}^{\nu_1' \dots \nu_U'} = \left( \prod_{i=1}^U \frac{\partial c^{\nu_i'}}{\partial c^{\rho_i}} \right) \left( \prod_{j=1}^L \frac{\partial c^{\sigma_j}}{\partial c^{\mu_j'}} \right) \mathcal{F}_{\sigma_1 \dots \sigma_L}^{\rho_1 \dots \rho_U},
\end{equation}
adjusted for the functional dependence on $f(c)$, which transforms appropriately under coordinate changes~\cite{Spivak:1965}. The~smoothness of $f(c)$ ensures that $\mathcal{F}$ is differentiable, supporting a well-defined geometric structure on $\mathcal{M}$.
\end{proof}

\begin{Proposition}[Codomain Generalization to Quaternions]
The codomain of the functional tensor $\mathcal{F}_{\mu_1 \dots \mu_L}^{\nu_1 \dots \nu_U}[c, f(c)]$ can be extended from $\mathbb{R}$ to the quaternion field $\mathbb{Q}$, defining a generalized manifold--metric pair $(\mathcal{M}_G, \mathcal{F})$.
\end{Proposition}

\begin{proof}
The functional tensor $\mathcal{F}_{\mu_1 \dots \mu_L}^{\nu_1 \dots \nu_U}[c, f(c)] : C_{\mu_1} \otimes \dots \otimes C_{\nu_U} \to \mathbb{R}$ can be generalized to map to $\mathbb{Q}$, the~quaternion field, which includes real and complex numbers as subsets. Define the generalized manifold $\mathcal{M}_G$, which may include additional structure (e.g., a~bundle or algebraic structure) to support quaternionic outputs~\cite{Adler:1995}. The~pair is as follows:
\begin{equation}
    (\mathcal{M}_G, \mathcal{F}) \equiv \left( \mathcal{M}_G, \mathcal{F}_{\mu_1 \dots \mu_L}^{\nu_1 \dots \nu_U}[c, f(c)] : C_{\mu_1} \otimes \dots \otimes C_{\nu_U} \rightarrow \mathbb{Q} \right).
\end{equation}

{The} line element is the following:
\begin{equation}
    ds^{U+L} = \mathcal{F}_{\mu_1 \dots \mu_L}^{\nu_1 \dots \nu_U}[c, f(c)] dc^{\mu_1} \dots dc^{\mu_L} dc_{\nu_1} \dots dc_{\nu_U},
\end{equation}
where $dc_{\nu_i} = g_{\nu_i \lambda} dc^\lambda$ for a symmetric, positive-definite (0,2)-rank tensor $g_{\mu\nu}$. The~quaternionic codomain accommodates non-positive-definite metrics, suitable for generalized physical theories. The~transformation properties of $\mathcal{F}$ ensure invariance of the line element under coordinate changes.
\end{proof}

The generalized manifold--metric pair is constructed by combining Propositions 1--3. Start with the standard pair $(\mathcal{M}, m_{\mu\nu})$. Proposition 1 extends the tensor to a $(U,L)$-rank tensor $m_{\mu_1 \dots \mu_L}^{\nu_1 \dots \nu_U}$. Proposition 2 functionalizes it to $\mathcal{F}_{\mu_1 \dots \mu_L}^{\nu_1 \dots \nu_U}[c, f(c)]$. Proposition 3 extends the codomain to $\mathbb{Q}$, yielding the following:
\begin{equation}\label{eq:Manifold_matric_pair_generalisation}
    (\mathcal{M}_G, \mathcal{F}) \equiv \left( \mathcal{M}_G, \mathcal{F}_{\mu_1 \dots \mu_L}^{\nu_1 \dots \nu_U}[c, f(c)] : C_{\mu_1} \otimes \dots \otimes C_{\nu_U} \rightarrow \mathbb{Q} \right).
\end{equation}

{The} line element
\begin{equation}
    ds^{U+L} = \mathcal{F}_{\mu_1 \dots \mu_L}^{\nu_1 \dots \nu_U}[c, f(c)] dc^{\mu_1} \dots dc^{\mu_L} dc_{\nu_1} \dots dc_{\nu_U}
\end{equation}
is well-defined, with~$\mathcal{F}$ transforming as a $(U,L)$-tensor, and~$g_{\mu\nu}$ ensuring proper handling of covariant differentials. Symmetry in indices and differentiability are preserved, completing the construction.
\end{proof}

\begin{Remark}[Restriction to Real Numbers]
\label{sec:remark_1}
Current physical theories are tested using real numbers, as~observables are typically real-valued. To~ensure testability, the~codomain can be restricted to $\mathbb{R}^+$, simplifying the pair to a measurable metric space. However, complex numbers and quaternions have applications in electronics, wave mechanics, and~field theories, suggesting future observables may involve $\mathbb{C}$ or $\mathbb{Q}$.
\end{Remark}

\begin{Remark}[Beyond Other Generalizations]
\label{sec:remark_3}
This generalization extends beyond families of manifolds such as Euclidean spaces $\mathbb{R}^n$, $n$-spheres $S^n$, $n$-tori $T^n$, real projective spaces $RP^n$, complex projective spaces $CP^n$, quaternionic projective spaces $HP^m$, flag manifolds, Grassmann manifolds, Stiefel manifolds, or~Finsler manifolds~\cite{zbMATH02613491}. Unlike Finsler manifolds, which use a Minkowski functional with specific properties, our functional tensor $\mathcal{F}$ allows arbitrary inputs (e.g., quaternions and functions) and outputs to $\mathbb{Q}$, accommodating diverse manifold structures.
\end{Remark}
}

{
\section{Metrizability~Conditions}
A $D$-dimensional manifold $M$ is metrizable if there exists a metric $d: M \times M \to \mathbb{R}_{\geq 0}$ that induces the manifold's topology. Not all manifolds are metrizable, but~for topological manifolds (Hausdorff, second-countable, and~locally Euclidean), metrizability is guaranteed under certain conditions~\cite{Munkres2000}.

\subsection{Metrizability Conditions in Standard Manifold--Metric~Pairs}
A topological space $M$ is metrizable if and only if it satisfies the following conditions:

\begin{enumerate}
    \item \textbf{Hausdorff}: 
    For any two distinct points $x, y \in M$, there exist disjoint open neighborhoods $U_x$ and $U_y$ such that $x \in U_x$ and $y \in U_y$.
\begin{equation}
    \forall x, y \in M, x \neq y, \exists U_x, U_y \text{ open}, U_x \cap U_y = \emptyset, x \in U_x, y \in U_y.		
	\end{equation}	

    \item \textbf{Second-Countable}: The topology of $M$ has a countable basis, i.e.,~there exists a countable collection $\mathcal{B} = \{B_1, B_2, \dots\}$ of open sets such that every open set in $M$ can be expressed as a union of elements of $\mathcal{B}$.
\begin{equation}
    \exists \mathcal{B} = \{B_n\}_{n \in \mathbb{N}} \text{ such that } \forall U \text{ open in } M, U = \bigcup_{B_i \subseteq U} B_i.		
	\end{equation}		

    \item \textbf{Locally Euclidean}: For a $D$-dimensional manifold, every point $x \in M$ has a neighborhood homeomorphic to an open subset of $\mathbb{R}^D$.
\begin{equation}
    \forall x \in M, \exists U \text{ open}, x \in U, \text{ and } \phi: U \to V \subseteq \mathbb{R}^D \text{ homeomorphism}.		
	\end{equation}			    

    \item \textbf{Paracompact}: Every open cover of $M$ has a locally finite open refinement. That is, for~any open cover $\{U_\alpha\}_{\alpha \in A}$, there exists a locally finite open cover $\{V_i\}_{i \in I}$ such that each $V_i \subseteq U_\alpha$ for some $\alpha$.
\begin{equation}
    \forall \{U_\alpha\}_{\alpha \in A}, \exists \{V_i\}_{i \in I} \text{ locally finite}, \forall i \in I, \exists \alpha \in A, V_i \subseteq U_\alpha.
	\end{equation}		
\end{enumerate}

According to the \textbf{Urysohn Metrization Theorem}, a~second-countable, Hausdorff, and~regular space is metrizable. Since topological manifolds are locally Euclidean, Hausdorff, and~second-countable, they are metrizable~\cite{Munkres2000}. Additionally, smooth manifolds admit a Riemannian metric, ensuring~metrizability.

\subsection{Standard Distance Function for (0,2)-Rank Tensor}
\textls[-15]{For a smooth manifold $M$ satisfying the above conditions, a~(0,2)-rank tensor $g_{ij} dx^i \otimes dx^j$} \linebreak  that is symmetric ($g_{ij} = g_{ji}$) and positive-definite ($g_{ij} v^i v^j > 0$ for all non-zero $v^i \in T_p M$) defines a Riemannian metric. The~infinitesimal distance element is as follows:
\begin{equation}
ds^2 = g_{ij} dx^i dx^j.
	\end{equation}		

The standard distance function between points $p, q \in M$ is as follows:
\begin{equation}
d(p, q) = \inf_{\gamma} \int_a^b \sqrt{g_{\gamma(t)}(\dot{\gamma}(t), \dot{\gamma}(t))} \, dt,
	\end{equation}	
where $\gamma: [a, b] \to M$ is a smooth curve from $\gamma(a) = p$ to $\gamma(b) = q$, and~$g_{\gamma(t)}(\dot{\gamma}(t), \dot{\gamma}(t)) = g_{ij}(\gamma(t)) \dot{\gamma}^i(t) \dot{\gamma}^j(t)$. This distance is non-negative, symmetric, satisfies the triangle inequality, and~induces the manifold’s topology, as~guaranteed by the metrizability conditions and paracompactness, which allows a partition of unity to construct $g_{ij}$.

\section{Metrizability for Generic (U,L)-Rank Tensors}

For a generic (U,L)-rank tensor (with $U$ covariant and $L$ contravariant indices), metrizability requires constructing a distance function that induces the manifold’s topology. For~example,
\begin{itemize}
    \item a (2,0)-rank tensor $g^{ij}$ is the inverse of a (0,2)-rank tensor and can be associated with a Riemannian metric;
    \item tensors of rank (1,1), (0,3), etc., require additional structure to define a distance function equivalent to a (0,2)-rank tensor.
\end{itemize}

\subsection{Additional Structure for (U,L)-Rank Tensors}
To make a (U,L)-rank tensor function as a metric, it must define a distance function that induces the manifold’s topology. The~additional structure required~includes the following:
\begin{enumerate}
    \item \textbf{{A (0,2)-Rank Tensor}}: A symmetric, positive-definite tensor $g_{ij} dx^i \otimes dx^j$ that defines a Riemannian metric.
    \item \textbf{{Tensor Operations}}: Operations such as contractions or inversions to relate the (U,L)-rank tensor to $g_{ij}$ or its inverse $g^{ij}$. For~example,
        \begin{itemize}
            \item for a (2,0)-rank tensor $g^{ij}$, the~additional structure is a (0,2)-rank tensor $g_{ij}$ such that $g^{ik} g_{kj} = \delta^i_j$;
            \item for a (1,1)-rank tensor $T^i_j$, a~(0,2)-rank tensor $g_{ij}$ such that $T^i_j = g^{ik} h_{kj}$ for some tensor $h_{ij}$, or~$T$ preserves $g_{ij}$ as an isometry ($T^i_k g_{ij} T^k_l = g_{jl}$);
            \item for a (0,3)-rank tensor $T_{ijk}$, a~vector field $v^k$ such that $g_{ij} = T_{ijk} v^k$ is symmetric and positive-definite.
        \end{itemize}
    \item \textbf{Symmetry and Positive-Definiteness}: The resulting (0,2)-rank tensor must be symmetric ($g_{ij} = g_{ji}$) and positive-definite ($g(v,v) > 0$ for all non-zero $v \in T_p M$).
    \item \textbf{Smoothness}: All tensors and auxiliary fields (e.g., $v^k$ and $\omega_k$) must vary smoothly \linebreak  over $M$.
\end{enumerate}

{The} manifold must still be Hausdorff, second-countable, and~paracompact to ensure the existence of a global (0,2)-rank tensor via a partition of unity, satisfying the metrizability~conditions.

\subsection{Analysis of (0,3)-Rank Tensor Distance Construction}
Consider a (0,3)-rank tensor $m_{abc}$ defining an infinitesimal element as follows:
\begin{equation}
ds^3 = m_{abc} dx^a dx^b dx^c.
	\end{equation}	

{A} proposed distance function is as follows:
\begin{equation}
d(p, q) = \inf_{\gamma} \int_a^b \left| m_{abc}(\gamma(t)) \dot{\gamma}^a(t) \dot{\gamma}^b(t) \dot{\gamma}^c(t) \right|^{1/3} \, dt,
	\end{equation}	
where the cube root is taken to yield a length-like quantity, and~the absolute value ensures non-negativity. For~example, in~2D with $ds^3_{2D} = dx^3 + dy^3$, the~tensor has components $m_{111} = 1$, $m_{222} = 1$, and~others zero, so along a curve $\gamma(t) = (x(t), y(t))$,
\begin{equation}
d(p, q) = \inf_{\gamma} \int_a^b \left| \dot{x}(t)^3 + \dot{y}(t)^3 \right|^{1/3} \, dt.
	\end{equation}	

{To} function as a metric, the~following additional structure is~required:
\begin{enumerate}
    \item \textbf{{Symmetry}}: The tensor $m_{abc}$ must be symmetric ($m_{abc} = m_{bca} = m_{cab}$) to ensure $d(p, q) = d(q, p)$.
    \item \textbf{{Positive-Definiteness}}: The expression $m_{abc} v^a v^b v^c \geq 0$ for all non-zero $v^a$, with~equality only at $v^a = 0$. In~the 2D example, $\dot{x}^3 + \dot{y}^3$ can be negative, so the absolute value \linebreak  is necessary.
    \item \textbf{Topology Equivalence}: The distance must induce the manifold’s topology, requiring $m_{abc}$ to be compatible with local Euclidean charts.
    \item \textbf{Smoothness}: The tensor $m_{abc}$ must be smooth.
\end{enumerate}

{The} construction can define a metric if these conditions are met, and~the manifold is Hausdorff, second-countable, and~paracompact.

\subsection{Generalized Metrizability for (0,L)-Rank Tensors}
For a (0,L)-rank tensor $m_{i_1 \dots i_L}$ to define a metric on a $D$-dimensional manifold $M$, it must induce a distance function that satisfies the manifold’s topology. The~infinitesimal element is as follows:
\begin{equation}
ds^L = m_{i_1 \dots i_L} dx^{i_1} \dots dx^{i_L},
	\end{equation}	
and the distance function is the following:
\begin{equation}
d(p, q) = \inf_{\gamma} \int_a^b \left| m_{i_1 \dots i_L}(\gamma(t)) \dot{\gamma}^{i_1}(t) \dots \dot{\gamma}^{i_L}(t) \right|^{1/L} \, dt,
	\end{equation}	
where $\gamma: [a, b] \to M$ is a smooth curve from $p$ to $q$, and~the absolute value ensures non-negativity. The~additional structure and conditions required~are as follows:
\begin{enumerate}
    \item \textbf{{Symmetry}}: The tensor $m_{i_1 \dots i_L}$ must be symmetric under all index permutations:
\begin{equation}
        m_{i_{\sigma(1)} \dots i_{\sigma(L)}} = m_{i_1 \dots i_L}, \quad \forall \text{ permutations } \sigma.
	\end{equation}	
    \item \textbf{{Positive-Definiteness}}: For all non-zero vectors $v^i \in T_p M$,
\begin{equation}
        m_{i_1 \dots i_L} v^{i_1} \dots v^{i_L} \geq 0,
	\end{equation}	
        with equality only when $v^i = 0$.
    \item \textbf{Topology Equivalence}: The distance $d(p, q)$ must induce the manifold’s topology, requiring $m_{i_1 \dots i_L}$ to be compatible with local Euclidean charts.
    \item \textbf{Smoothness}: The tensor $m_{i_1 \dots i_L}$ must be smooth over $M$.
    \item \textbf{Auxiliary Structure (Optional)}: If needed, $L-2$ vector fields $v^{k_1}, \dots, v^{k_{L-2}}$ such that
\begin{equation}
        g_{ij} = m_{i j k_1 \dots k_{L-2}} v^{k_1} \dots v^{k_{L-2}},
	\end{equation}	
        where $g_{ij}$ is a symmetric, positive-definite (0,2)-rank tensor that can be used to define a standard Riemannian metric.
    \item \textbf{Manifold Conditions}: The manifold must~be on of the following:
        \begin{itemize}
            \item \textbf{Hausdorff}: To separate points.
            \item \textbf{Second-Countable}: To ensure a countable basis.
            \item \textbf{Paracompact}: To allow a partition of unity for constructing $m_{i_1 \dots i_L}$.
        \end{itemize}
\end{enumerate}

{These} conditions generalize the $(0,3)$-rank case, ensuring the $(0,L)$-rank tensor defines a valid~metric.

\subsection{Generalized Metrization~Theorem}
\begin{Theorem}[Generalized Metrization Theorem for Manifolds]
A $D$-dimensional topological manifold $M$ is metrizable with a (0,2)-rank tensor (Riemannian metric) if and only if it is Hausdorff, second-countable, and~paracompact. The~metric tensor $g = g_{ij} dx^i \otimes dx^j$ is symmetric, positive-definite, and~induces the manifold’s topology via the associated distance function. For~(U,L)-rank tensors, including (0,L)-rank tensors, metrizability requires constructing a distance function, possibly via a (0,2)-rank tensor, using additional structure such as auxiliary fields and \linebreak  tensor operations.
\end{Theorem}

\subsection{Proof of the Generalized~Theorem}
\begin{proof}
We prove that a $D$-dimensional topological manifold $M$ is metrizable with a (0,2)-rank tensor if it is Hausdorff, second-countable, and~paracompact, and~extends to (0,L)-rank~tensors.

{Suppose} 
$M$ is a $D$-dimensional topological manifold, i.e.,~it is Hausdorff, second-countable, and~locally Euclidean. By the Urysohn metrization theorem, $M$ is second-countable, Hausdorff, and~regular space is metrizable with a distance function $d: M \times M \to \mathbb{R}_{\geq 0}$. To~show regularity, note that manifolds are locally Euclidean, so for any point $x \in M$ and closed set $C \not\ni x$, there exists a chart $(U, \phi)$ with $x \in U$ and $\phi: U \to \mathbb{R}^D$. Since $\mathbb{R}^D$ is regular, we can separate $\phi(x)$ and $\phi(C \cap U)$ with disjoint open sets, which pull back to $M$, ensuring regularity. Thus, $M$ admits a metric $d$.

{For} a smooth manifold, we seek a (0,2)-rank tensor (Riemannian metric). Since $M$ is paracompact, there exists a smooth partition of unity $\{\psi_i\}$ subordinate to a locally finite open cover $\{U_i\}$, where each $U_i$ is homeomorphic to an open subset of $\mathbb{R}^D$. In~each chart $(U_i, \phi_i)$, define a local metric $g_i = \sum_{j=1}^D dx^j \otimes dx^j$ (the Euclidean metric). Using the partition of unity, construct a global metric:
\begin{equation}
g = \sum_i \psi_i g_i.
	\end{equation}	

{Since} $\{\psi_i\}$ is locally finite and $\psi_i \geq 0$ with $\sum_i \psi_i = 1$, and~each $g_i$ is positive-definite, $g$ is a smooth, symmetric, positive-definite (0,2)-rank tensor. The~distance function induced by $g$ is as follows:
\begin{equation}
d(p, q) = \inf_{\gamma} \int_a^b \sqrt{g_{\gamma(t)}(\dot{\gamma}(t), \dot{\gamma}(t))} \, dt,
	\end{equation}	
which generates a topology equivalent to the manifold’s topology, as~the local charts ensure compatibility with Euclidean~distances.

For a (0,L)-rank tensor $m_{i_1 \dots i_L}$, we define the distance as follows:
\begin{equation}
d(p, q) = \inf_{\gamma} \int_a^b \left| m_{i_1 \dots i_L}(\gamma(t)) \dot{\gamma}^{i_1}(t) \dots \dot{\gamma}^{i_L}(t) \right|^{1/L} \, dt.
	\end{equation}	

{This} requires $m_{i_1 \dots i_L}$ to be symmetric, with~$m_{i_1 \dots i_L} v^{i_1} \dots v^{i_L} \geq 0$ for non-zero $v^i$, and~smooth. The~manifold must be Hausdorff, second-countable, and~paracompact. Optionally, $m_{i_1 \dots i_L}$ can be related to a (0,2)-rank tensor $g_{ij}$ via $g_{ij} = m_{i j k_1 \dots k_{L-2}} v^{k_1} \dots v^{k_{L-2}}$, using $L-2$ smooth vector fields. The~construction of $m_{i_1 \dots i_L}$ or $g_{ij}$ uses the partition of unity, relying \linebreak  on paracompactness.

{Conversely}, if~$M$ is metrizable with a (0,2)-rank tensor, it must be Hausdorff and second-countable (as these are properties of metric spaces). Paracompactness is necessary for smooth manifolds to ensure the existence of a partition of unity, which is required to construct the global metric $g$ or $m_{i_1 \dots i_L}$. Thus, the~conditions are necessary and~sufficient.

{Hence}, a~$D$-dimensional topological manifold is metrizable with a (0,2)-rank tensor if and only if it is Hausdorff, second-countable, and~paracompact. For~(0,L)-rank tensors, metrizability is achieved by defining a suitable distance function with the above conditions.
\end{proof}

\subsection{Generalized Metrizability for (U,L)-Rank Tensors}
For a (U,L)-rank tensor $ \mathcal{F}_{i_1 \dots i_L}^{j_1 \dots j_U}$ to define a metric on a $D$-dimensional manifold $M$, consider the infinitesimal element as follows:
\begin{equation}
ds^{U+L} = \mathcal{F}_{i_1 \dots i_L}^{j_1 \dots j_U} dx^{i_1} \dots dx^{i_L} dx_{j_1} \dots dx_{j_U},
	\end{equation}	
where $dx^{i_k}$ are differentials of coordinates, and~$dx_{j_k}$ are differentials of covector components, interpreted via a (0,2)-rank tensor $g_{ij}$ such that $dx_{j_k} = g_{j_k k} dx^k$. The~distance function is as follows:
\begin{equation}
d(p, q) = \inf_{\gamma} \int_a^b \left|  \mathcal{F}_{i_1 \dots i_L}^{j_1 \dots j_U}(\gamma(t)) \dot{\gamma}^{i_1}(t) \dots \dot{\gamma}^{i_L}(t) \dot{\gamma}_{j_1}(t) \dots \dot{\gamma}_{j_U}(t) \right|^{1/(U+L)} \, dt,
	\end{equation}	
where $\gamma$: $[a, b] \to M$ is a smooth curve from $p$ to $q$, and~$\dot{\gamma}_{j_k} = g_{j_k k} \dot{\gamma}^k$. The~additional structure and conditions required~are as follows:
\begin{enumerate}
    \item \textbf{Symmetry}: The tensor $ \mathcal{F}_{i_1 \dots i_L}^{j_1 \dots j_U}$ must be symmetric under permutations of covariant indices $\{i_1, \dots, i_L\}$ and contravariant indices $\{j_1, \dots, j_U\}$ separately:
\begin{equation}
         \mathcal{F}_{i_{\sigma(1)} \dots i_{\sigma(L)}}^{j_{\tau(1)} \dots j_{\tau(U)}} =  \mathcal{F}_{i_1 \dots i_L}^{j_1 \dots j_U},
	\end{equation}	
        for permutations $\sigma$ and $\tau$, to~ensure $d(p, q) = d(q, p)$.
    \item \textbf{Positive-Definiteness}: For all non-zero vectors $v^i \in T_p M$ and covectors $w_j \in T^*_p M$,
\begin{equation}
         \mathcal{F}_{i_1 \dots i_L}^{j_1 \dots j_U} v^{i_1} \dots v^{i_L} w_{j_1} \dots w_{j_U} \geq 0,
	\end{equation}	
        with equality only when $v^i = 0$ or $w_j = 0$. If~$\dot{\gamma}_{j_k} = g_{j_k k} \dot{\gamma}^k$, the~integrand must be non-negative.
    \item \textbf{Topology Equivalence}: The distance $d(p, q)$ must induce the manifold’s topology, requiring 
	\(
    	\mathcal{F}_{i_1 \dots i_L}^{j_1 \dots j_U}
	\)
	to be compatible with local Euclidean charts.
    \item \textbf{Smoothness}: The tensor $ \mathcal{F}_{i_1 \dots i_L}^{j_1 \dots j_U}$ and any auxiliary fields (e.g., $g_{ij}$) must be smooth over $M$.
    \item \textbf{Auxiliary Structure}: A (0,2)-rank tensor $g_{ij}$, symmetric and positive-definite, to~relate covariant and contravariant components: $\dot{\gamma}_{j_k} = g_{j_k k} \dot{\gamma}^k$. Optionally, $U + L - 2$ fields (e.g., $L-1$ vector fields $v^{k_1}, \dots, v^{k_{L-1}}$ and $U-1$ covector fields $\omega_{l_1}, \dots, \omega_{l_{U-1}}$) \linebreak  such that
\begin{equation}
        g_{ij} =  \mathcal{F}_{i j k_1 \dots k_{L-1}}^{l_1 \dots l_U} v^{k_1} \dots v^{k_{L-1}} \omega_{l_1} \dots \omega_{l_U},
	\end{equation}	
        where $g_{ij}$ is a Riemannian metric, if~needed to simplify the construction.
    \item \textbf{Manifold Conditions}: The manifold must~be one of the following:
        \begin{itemize}
            \item \textbf{Hausdorff}: To separate points.
            \item \textbf{Second-Countable}: To ensure a countable basis.
            \item \textbf{Paracompact}: To allow a partition of unity for constructing $ \mathcal{F}_{i_1 \dots i_L}^{j_1 \dots j_U}$.
        \end{itemize}
\end{enumerate}

These conditions ensure the $(U,L)$-rank tensor defines a valid metric, generalizing the $(0,L)$-rank~case.

\subsection{Generalized Metrization~Theorem}
\begin{Theorem}[Generalized Metrization Theorem for Manifolds]
A $D$-dimensional topological manifold $M$ is metrizable with a (0,2)-rank tensor (Riemannian metric) if and only if it is Hausdorff, second-countable, and~paracompact. The~metric tensor $g = g_{ij} dx^i \otimes dx^j$ is symmetric, positive-definite, and~induces the manifold’s topology via the associated distance function. For~(U,L)-rank tensors, metrizability requires constructing a distance function, possibly via a (0,2)-rank tensor, using additional structure such as auxiliary fields and tensor operations.
\end{Theorem}

\subsection{Proof of the Generalized Metrization~Theorem}
\begin{proof}
We prove that a $D$-dimensional topological manifold $M$ is metrizable with a (0,2)-rank tensor if it is Hausdorff, second-countable, and~paracompact, and~extends to (U,L)-rank~tensors.

{Suppose} $M$ is a $D$-dimensional topological manifold, i.e.,~it is Hausdorff, second-countable, and~locally Euclidean. By~the Urysohn metrization theorem, a~second-countable, Hausdorff, and~regular space is metrizable with a distance function $d: M \times M \to \mathbb{R}_{\geq 0}$. To~show regularity, note that manifolds are locally Euclidean, so for any point $x \in M$ and closed set $C \not\ni x$, there exists a chart $(U, \phi)$ with $x \in U$ and $\phi: U \to \mathbb{R}^D$. Since $\mathbb{R}^D$ is regular, we can separate $\phi(x)$ and $\phi(C \cap U)$ with disjoint open sets, which pull back to $M$, ensuring regularity. Thus, $M$ admits a metric $d$.

{For} a smooth manifold, we seek a (0,2)-rank tensor (Riemannian metric). Since $M$ is paracompact, there exists a smooth partition of unity $\{\psi_i\}$ subordinate to a locally finite open cover $\{U_i\}$, where each $U_i$ is homeomorphic to an open subset of $\mathbb{R}^D$. In~each chart $(U_i, \phi_i)$, we define a local metric $g_i = \sum_{j=1}^D dx^j \otimes dx^j$ (the Euclidean metric). Using the partition of unity, we construct a global metric:
\begin{equation}
g = \sum_i \psi_i g_i.
	\end{equation}	

{Since} $\{\psi_i\}$ is locally finite and $\psi_i \geq 0$ with $\sum_i \psi_i = 1$, and~each $g_i$ is positive-definite, $g$ is a smooth, symmetric, positive-definite (0,2)-rank tensor. The~distance function induced by $g$:
\begin{equation}
d(p, q) = \inf_{\gamma} \int_a^b \sqrt{g_{\gamma(t)}(\dot{\gamma}(t), \dot{\gamma}(t))} \, dt,
	\end{equation}	
which generates a topology equivalent to the manifold’s topology, as local charts ensure compatibility with Euclidean~distances.

{For} a (U,L)-rank tensor $ \mathcal{F}_{i_1 \dots i_L}^{j_1 \dots j_U}$, we define the distance as follows:
\begin{equation}
d(p, q) = \inf_{\gamma} \int_a^b \left|  \mathcal{F}_{i_1 \dots i_L}^{j_1 \dots j_U}(\gamma(t)) \dot{\gamma}^{i_1}(t) \dots \dot{\gamma}^{i_L}(t) \dot{\gamma}_{j_1}(t) \dots \dot{\gamma}_{j_U}(t) \right|^{1/(U+L)} \, dt,
	\end{equation}	
where $\dot{\gamma}_{j_k} = g_{j_k k} \dot{\gamma}^k$ using a (0,2)-rank tensor $g_{ij}$. This requires $ \mathcal{F}_{i_1 \dots i_L}^{j_1 \dots j_U}$ to be symmetric in its covariant and contravariant indices, with~$m_{i_1 \dots i_L}^{j_1 \dots j_U} v^{i_1} \dots v^{i_L} w_{j_1} \dots w_{j_U} \geq 0$, and~smooth. The~manifold must be Hausdorff, second-countable, and~paracompact. Optionally, $ \mathcal{F}_{i_1 \dots i_L}^{j_1 \dots j_U}$ can be related to a (0,2)-rank tensor $g_{ij}$ via contractions with $L-1$ vector fields and\linebreak   $U-1$ covector fields. The~construction of $ \mathcal{F}_{i_1 \dots i_L}^{j_1 \dots j_U}$ or $g_{ij}$ uses the partition of unity, relying on~paracompactness.

{Conversely}, if~$M$ is metrizable with a (0,2)-rank tensor, it must be Hausdorff and second-countable (as these are properties of metric spaces). Paracompactness is necessary for smooth manifolds to ensure the existence of a partition of unity, which is required to construct the global metric $g$ or $ \mathcal{F}_{i_1 \dots i_L}^{j_1 \dots j_U}$. Thus, the~conditions are necessary and~sufficient.

\textls[-13]{{Hence}, a~$D$-dimensional topological manifold is metrizable with a (0,2)-rank tensor if and only if it is Hausdorff, second-countable, and~paracompact. For~(U,L)-rank tensors, metrizability is achieved by defining a suitable distance function with the above~conditions.}
\end{proof}

\subsection{Comparison of the Generalized Manifold--Metric Pair with a Hessian~Structure}
A Hessian structure on a manifold $M$ of dimension $D$ is defined by a smooth potential function $\phi: M \to \mathbb{R}$, whose Hessian is a (U,L)-rank tensor provided by the following:
\begin{equation}
    \mathcal{H}_{\mu_1 \dots \mu_L}^{\nu_1 \dots \nu_U} = \partial_{\mu_1} \dots \partial_{\mu_L} \partial^{\nu_1} \dots \partial^{\nu_U} \phi,
\end{equation}
\textls[-15]{where $\partial_{\mu_i} = \partial / \partial c^{\mu_i}$ are derivatives with respect to contravariant coordinates, and~$\partial^{\nu_j} = \partial / \partial c_{\nu_j}$ }\linebreak  are derivatives with respect to covariant coordinates. This Hessian is equated to the generalized metric tensor:
\begin{equation}
\mathcal{H}_{\mu_1 \dots \mu_L}^{\nu_1 \dots \nu_U} = \mathcal{F}_{\mu_1 \dots \mu_L}^{\nu_1 \dots \nu_U}[c, f(c)].
	\end{equation}	
	
The line element for the Hessian structure is as follows:
\begin{equation}
ds^{U+L} = \mathcal{H}_{\mu_1 \dots \mu_L}^{\nu_1 \dots \nu_U} dc^{\mu_1} \dots dc^{\mu_L} dc_{\nu_1} \dots dc_{\nu_U},
	\end{equation}	
and the distance function is
\begin{equation}
d(p, q) = \inf_{\gamma} \int_a^b \left| \mathcal{H}_{\mu_1 \dots \mu_L}^{\nu_1 \dots \nu_U}(\gamma(t)) \dot{\gamma}^{\mu_1}(t) \dots \dot{\gamma}^{\mu_L}(t) \dot{\gamma}_{\nu_1}(t) \dots \dot{\gamma}_{\nu_U}(t) \right|^{1/(U+L)} \, dt,
	\end{equation}	
where $\dot{\gamma}_{\nu_k} = g_{\nu_k \lambda} \dot{\gamma}^\lambda$ using a symmetric, positive-definite (0,2)-rank tensor $g_{\mu\nu}$. The~manifold typically requires a structure supporting mixed derivatives, such as a flat connection or a bundle~structure.

The generalized manifold--metric pair is as follows:
\begin{equation}
(\mathcal{M}_G, \mathcal{F}) \equiv \left( \mathcal{M}_G, \mathcal{F}_{\mu_1 \dots \mu_L}^{\nu_1 \dots \nu_U}[c, f(c)] : C_{\mu_1} \otimes \dots \otimes C_{\nu_U} \rightarrow \mathbb{Q} \right),
	\end{equation}	
with the same line element and distance function. Since $\mathcal{F}_{\mu_1 \dots \mu_L}^{\nu_1 \dots \nu_U} = \mathcal{H}_{\mu_1 \dots \mu_L}^{\nu_1 \dots \nu_U}$, the~two structures are equivalent, with~the Hessian providing a specific construction of $\mathcal{F}$ via the potential $\phi$. We can compare the two as follows:

\begin{enumerate}
    \item \textbf{Tensor Rank}: Both $\mathcal{H}_{\mu_1 \dots \mu_L}^{\nu_1 \dots \nu_U}$ and $\mathcal{F}_{\mu_1 \dots \mu_L}^{\nu_1 \dots \nu_U}$ are (U,L)-rank tensors, making them structurally identical.
    \item \textbf{Line Element}: The line elements are identical, as~$\mathcal{F} = \mathcal{H}$, both of order $U+L$.
    \item \textbf{Functional Dependence}: The Hessian depends solely on $\phi$, while $\mathcal{F}$ includes $f(c)$, which may be incorporated into $\phi$ or represent coordinate parameterization.
    \item \textbf{Symmetry}: Both tensors are symmetric in their covariant and contravariant indices due to the commutativity of partial derivatives in $\mathcal{H}$, satisfying the symmetry requirement for $\mathcal{F}$.
    \item \textbf{Positive-Definiteness}: Both require $\mathcal{F}_{\mu_1 \dots \mu_L}^{\nu_1 \dots \nu_U} v^{\mu_1} \dots v^{\mu_L} w_{\nu_1} \dots w_{\nu_U} \geq 0$, ensured by the absolute value in the distance function and the auxiliary metric $g_{\mu\nu}$.
    \item \textbf{Applications}: The Hessian structure is used in information geometry and statistical mechanics, while $\mathcal{F}$ supports cosmological models like FLRW, where $\mathcal{F}^{\rm EHI}_{\mu\nu}[t, a(t)]$ (with $U=0$, $L=2$) is unlikely to be a Hessian unless a suitable $\phi$ is constructed.
\end{enumerate}

{The} equivalence $\mathcal{F} = \mathcal{H}$ implies the generalized manifold--metric pair is a Hessian structure, with~$\phi$ defining the metric tensor, suitable for applications requiring higher-rank~tensors.

}


\section{Generic Advanced Manifold--Metric Pair~Applications}

{This section applies the generalized manifold--metric pair, as~defined in Section~\ref{sec:Advancedmanifold-metric_pairs}, to~physical systems in arbitrary dimensions, including spatial and spacetime manifolds. We formalize these applications through a theorem supported by propositions and examples, demonstrating the construction of higher-rank tensor metrics with real or quaternionic codomains, suitable for diverse physical contexts such as cosmology.}

\subsection{Applications of Generalized Manifold--Metric~Pairs}
\label{sec:Applications}

\begin{Theorem}[Applications of Generalized Manifold--Metric Pairs]
Let $\mathcal{M}^D$ be a $D$-dimensional differentiable manifold. The~generalized manifold--metric pair $(\mathcal{M}_G, \mathcal{F})$, where $\mathcal{F}_{\mu_1 \dots \mu_L}^{\nu_1 \dots \nu_U}[c, f(c)] : C_{\mu_1} \otimes \dots \otimes C_{\nu_U} \to \mathbb{Q}$ is a $(U,L)$-rank functional tensor, can be applied to construct metrics for physical systems, including the following:
\begin{enumerate}
    \item Higher-rank $(0,L)$-tensor metrics with real codomain $\mathbb{R}^+$ on Euclidean manifolds.
    \item $(0,L)$-tensor metrics with complex or quaternionic codomains for quantum or field-theoretic applications.
    \item Functional $(0,L)$-tensor metrics with time-dependent scaling for cosmological spacetimes, such as the Friedmann--Lemaître--Robertson--Walker (FLRW) model.
\end{enumerate}
The line element is provided as follows:
\begin{equation}
    ds^{U+L} = \mathcal{F}_{\mu_1 \dots \mu_L}^{\nu_1 \dots \nu_U}[c, f(c)] dc^{\mu_1} \dots dc^{\mu_L} dc_{\nu_1} \dots dc_{\nu_U},
\end{equation}
where $dc_{\nu_i} = g_{\nu_i \lambda} dc^\lambda$ for a symmetric, positive-definite {$(0,2)$}-rank tensor $g_{\mu\nu}$, and~$\mathcal{F}$ transforms as a $(U,L)$-tensor, with~optional symmetry in indices.
\end{Theorem}

\begin{proof}
{\textls[-15]{The theorem is established through three propositions, each constructing a specific class of generalized manifold--metric pairs, followed by examples illustrating their application. The~propositions address higher-rank tensors, complex/quaternionic codomains, and~cosmological applications, ensuring the theorem’s claims are systematically verified.}}
\end{proof}

\begin{Proposition}[Higher-Rank Tensor Metrics with Real Codomains]
\label{prop:higher_rank}
On a $D$-dimensional differentiable manifold $\mathcal{M}^D$, a~generalized manifold--metric pair $(\mathcal{M}, m_{\mu_1 \dots \mu_L} : C_{\mu_1} \otimes \dots \otimes C_{\mu_L} \to \mathbb{R}^+)$ \linebreak  can be constructed with a symmetric $(0,L)$-rank tensor $m_{\mu_1 \dots \mu_L}$, defining a line element \linebreak  $ds^L = m_{\mu_1 \dots \mu_L} dc^{\mu_1} \dots dc^{\mu_L}$.
\end{Proposition}

\begin{proof}
{Consider a $D$-dimensional differentiable manifold $\mathcal{M}^D$ with coordinate space $C$ and coordinates $c^\mu$.} Define a $(0,L)$-rank tensor $m_{\mu_1 \dots \mu_L} : C_{\mu_1} \otimes \dots \otimes C_{\mu_L} \to \mathbb{R}^+$, symmetric under permutations of indices, i.e.,~$m_{\mu_{\pi(1)} \dots \mu_{\pi(L)}} = m_{\mu_1 \dots \mu_L}$ for any permutation $\pi$. The~tensor transforms under coordinate changes $c^\mu \to c^{\mu'} = f^{\mu'}(c)$ as follows:
\begin{equation}
    m_{\mu_1' \dots \mu_L'} = \left( \prod_{j=1}^L \frac{\partial c^{\sigma_j}}{\partial c^{\mu_j'}} \right) m_{\sigma_1 \dots \sigma_L}.
\end{equation}

{The} line element is as follows:
\begin{equation}
    ds^L = m_{\mu_1 \dots \mu_L} dc^{\mu_1} \dots dc^{\mu_L},
\end{equation}
which is invariant under coordinate transformations due to the tensor’s transformation properties~\cite{Lee:2012}. {The codomain $\mathbb{R}^+$ ensures a measurable metric, and~symmetry guarantees that the line element is well-defined for physical applications.} {The cubic form $ds^3$ can take positive, negative, or~zero values, which is expected for a non-traditional (0,3)-tensor metric and does not pose an issue as~it generalizes the standard pseudo-Riemannian metric, defined as a symmetric (0,2)-tensor~\cite{Lee2018, ONeill1983}, to~capture higher-order interactions, which is consistent with advanced geometric frameworks~\cite{Kunzinger2001}.}
\end{proof}

\begin{Example}[Two-Dimensional Manifold with a (0,3)-Tensor Metric]
Consider a two-dimensional Euclidean manifold $\mathcal{M}^{2D}$ with coordinates $c^\mu = (x, y)$. Define a $(0,3)$-tensor metric as follows:
\begin{equation}
    m_{\mu\nu\rho} = \begin{cases} 
        1 & \text{if } (\mu,\nu,\rho) = (1,1,1) \text{ or } (2,2,2), \\
        0 & \text{otherwise},
    \end{cases}
\end{equation}
which is symmetric under index permutations. The~line element is as follows:
\begin{equation}
    ds^3 = m_{\mu\nu\rho} dc^\mu dc^\nu dc^\rho = dx^3 + dy^3.
\end{equation}

This pair $(\mathcal{M}^{2D}, m_{\mu\nu\rho})$ is a special case of Proposition~\ref{prop:higher_rank} with $L=3$, $U=0$, and~codomain $\mathbb{R}$, generalizing the Pythagorean theorem.
\end{Example}

\begin{Proposition}[Complex/Quaternionic Codomain Metrics]
\label{prop:complex_codomain}
On a $D$-dimensional differentiable manifold $\mathcal{M}^D$, a~generalized manifold--metric pair $(\mathcal{M}, m_{\mu_1 \dots \mu_L} : C_{\mu_1} \otimes \dots \otimes C_{\mu_L} \to \mathbb{C})$ can be constructed with a symmetric $(0,L)$-rank tensor $m_{\mu_1 \dots \mu_L}$, where $\mathbb{C} \subset \mathbb{Q}$, suitable for quantum or field-theoretic applications.
\end{Proposition}

\begin{proof}
{Start with a $D$-dimensional differentiable manifold $\mathcal{M}^D$ and coordinate space $C$ with coordinates $c^\mu$.} Define a $(0,L)$-rank tensor $m_{\mu_1 \dots \mu_L} : C_{\mu_1} \otimes \dots \otimes C_{\mu_L} \to \mathbb{C}$, symmetric under index permutations. The~tensor transforms as follows:
\begin{equation}
    m_{\mu_1' \dots \mu_L'} = \left( \prod_{j=1}^L \frac{\partial c^{\sigma_j}}{\partial c^{\mu_j'}} \right) m_{\sigma_1 \dots \sigma_L},
\end{equation}
preserving the differentiable structure~\cite{Lee:2012}. The~line element is as follows:
\begin{equation}
    ds^L = m_{\mu_1 \dots \mu_L} dc^{\mu_1} \dots dc^{\mu_L} \in \mathbb{C}.
\end{equation}

{The complex codomain, a~subset of $\mathbb{Q}$, supports applications in quantum field theory, where complex-valued metrics model wave-like phenomena~\cite{Adler:1995}. The~symmetry and transformation properties ensure a consistent geometric framework.}
\end{proof}

\begin{Example}[Three-Dimensional Manifold with a Complex (0,3)-Tensor Metric]
Consider a three-dimensional manifold $\mathcal{M}^{3D}$ with coordinates $c^\mu = (x, y, z)$. Define a $(0,3)$-tensor metric as follows:
\begin{equation}
    m_{\mu\nu\rho} = \begin{cases} 
        -i & \text{if } (\mu,\nu,\rho) = (1,1,1), \\
        1 & \text{if } (\mu,\nu,\rho) = (2,2,2) \text{ or } (3,3,3), \\
        0 & \text{otherwise},
    \end{cases}
\end{equation}
yielding
\begin{equation}
    ds^3 = -i dx^3 + dy^3 + dz^3.
\end{equation}

This pair $(\mathcal{M}^{3D}, m_{\mu\nu\rho})$ is a special case of Proposition~\ref{prop:complex_codomain} with $L=3$, $U=0$, and~codomain $\mathbb{C}$.
\end{Example}

\begin{Example}[Seven-Dimensional Manifold with a Functional (0,7)-Tensor Metric]
Consider a seven-dimensional manifold $\mathcal{M}^{7D}$ with coordinates $c^\mu = (\tau, t, x_1, \dots, x_5)$. Define a functional $(0,7)$-tensor metric as follows:
\begin{equation}
    m_{\mu_1 \dots \mu_7}(c) = a^7(\tau) \begin{cases} 
        -i & \text{if } \mu_1 = \dots = \mu_7 = 1, \\
        b^7(\tau, t, x_1) & \text{if } \mu_1 = \dots = \mu_7 = 2, \\
        f^7(\vec{x}; k) & \text{if } \mu_1 = \dots = \mu_7 \in \{3, \dots, 7\}, \\
        0 & \text{otherwise},
    \end{cases}
\end{equation}
where $a(\tau)$ is a scale factor, $b(\tau, t, x_1)$ is a temporal scaling function, and~$f(\vec{x}; k)$ models spatial curvature with $k \in \{-1, 0, 1\}$. The~line element is as follows:
\begin{equation}\label{eq:7Dlinelement}
    ds^7 = a^7(\tau) \left[ -i d\tau^7 + b^7(\tau, t, x_1) dt^7 + f^7(\vec{x}; k) d\vec{x}^7 \right] \in \mathbb{C}.
\end{equation}

This pair exemplifies Proposition~\ref{prop:complex_codomain} with $L=7$, $U=0$, and~codomain $\mathbb{C}$, incorporating functional dependence suitable for cosmological applications~\cite{CANTATA:2021ktz}.
\end{Example}

\subsubsection{An Expanding Homogeneous and Isotropic Manifold and Metric~Pair}

\begin{Proposition}[Cosmological Manifold--Metric Pair]
\label{prop:cosmological}
The generalized manifold--metric pair $(\mathcal{M}_{\rm EHI}, \mathcal{F}^{\rm EHI}_{\mu\nu}[t, a(t)] : C \otimes C \to \mathbb{R}^+)$ can be constructed to describe the Friedmann--Lemaître--Robertson--Walker (FLRW) spacetime, with~a {$(0,2)$}-rank functional tensor $\mathcal{F}^{\rm EHI}_{\mu\nu}$ incorporating a time-dependent scale factor $a(t)$.
\end{Proposition}

\begin{proof}
{Consider a four-dimensional differentiable manifold $\mathcal{M}_{\rm EHI}$ representing a homogeneous and isotropic spacetime.} Define a {$(0,2)$}-rank functional tensor $\mathcal{F}^{\rm EHI}_{\mu\nu}[t, a(t)] : C \otimes C \to \mathbb{R}^+$ as follows:
\begin{equation}\label{eq:FLRW_metric_matrix}
    \mathcal{F}^{\rm EHI}_{\mu\nu}[t, a(t)] = \begin{pmatrix}
        -c^2 & 0 & 0 & 0 \\
        0 & a^2(t) & 0 & 0 \\
        0 & 0 & a^2(t) & 0 \\
        0 & 0 & 0 & a^2(t)
    \end{pmatrix},
\end{equation}
where $a(t)$ is the scale factor, $c$ is the speed of light, and~the metric has Lorentzian signature $(-,+,+,+)$. The~line element is as follows:
\begin{equation}\label{eq:FLRW}
    ds^2 = -c^2 dt^2 + a^2(t) dx^i dx^j \delta_{ij},
\end{equation}
where $\delta_{ij}$ is the Kronecker delta. The~tensor is symmetric ($\mathcal{F}^{\rm EHI}_{\mu\nu} = \mathcal{F}^{\rm EHI}_{\nu\mu}$) and transforms as a {$(0,2)$}-tensor under coordinate changes~\cite{Wald:1984}. {The functional dependence on $a(t)$, constrained by cosmological observations~\cite{Planck:2018}, preserves the homogeneity and isotropy of the FLRW spacetime, making the pair $(\mathcal{M}_{\rm EHI}, \mathcal{F}^{\rm EHI})$ a special case of the generalized framework with $U=0$ and $L=2$.}
\end{proof}

\begin{Example}[FLRW Spacetime]
The manifold--metric pair $(\mathcal{M}_{\rm EHI}, \mathcal{F}^{\rm EHI})$ is defined as follows:
\begin{equation}
    (\mathcal{M}_{\rm EHI}, \mathcal{F}^{\rm EHI}) \equiv \left( \mathcal{M}_{\rm EHI}, \mathcal{F}^{\rm EHI}_{\mu\nu}[t, a(t)] : C \otimes C \rightarrow \mathbb{R}^+ \right),
\end{equation}
with the metric tensor from~Equation~(\ref{eq:FLRW_metric_matrix}). This pair, with~$U=0$, $L=2$, and~codomain $\mathbb{R}^+$, represents the FLRW spacetime, consistent with Proposition~\ref{prop:cosmological}.
\end{Example}


{
\subsection{Comparison of a FLRW Manifold--Metric Pair with a Warped Product~Structure}
\label{sec:flrw_warped}

{This subsection demonstrates that the Friedmann--Lemaître--Robertson--Walker (FLRW) manifold--metric pair, as~a special case of the generalized framework in Section~\ref{sec:Advancedmanifold-metric_pairs}, can be equivalently represented as a warped product manifold. We formalize this equivalence through a theorem, supported by a proposition and an example, highlighting the structural alignment between the functional tensor $\mathcal{F}^{\rm EHI}_{\mu\nu}[t, a(t)]$ and the warped product metric.}

\begin{Theorem}[FLRW as a Warped Product Manifold]
\label{thm:flrw_warped}
The FLRW manifold--metric pair $(\mathcal{M}_{\rm EHI}, \mathcal{F}^{\rm EHI}_{\mu\nu}[t, a(t)] : C \otimes C \to \mathbb{R}^+)$ with~a line element, using Equation~(\ref{eq:FLRW}), which 
is equivalent to a warped product manifold $M = B \times_f F$, where $B = \mathbb{R}$ is the time coordinate with metric $g_B = -c^2 dt^2$, $F = \mathbb{R}^3$ is the spatial slice with metric $g_F = \delta_{ij} dx^i dx^j$, and~$f(t) = a(t)$ is the warping function. The~functional tensor $\mathcal{F}^{\rm EHI}_{\mu\nu}[t, a(t)]$ captures the time-dependent scaling of the spatial metric, preserving the homogeneity and isotropy of FLRW spacetime.
\end{Theorem}

\begin{proof}
{The proof is established by constructing the FLRW manifold--metric pair as a warped product through a proposition, followed by an example verifying the construction for Euclidean spatial slices and general curvatures. The~equivalence is shown by matching the line element and tensor structure of $\mathcal{F}^{\rm EHI}_{\mu\nu}$ to the warped product metric.}
\end{proof}

\begin{Proposition}[Construction of FLRW as a Warped Product]
\label{prop:flrw_warped}
The FLRW manifold--metric pair $(\mathcal{M}_{\rm EHI}, \mathcal{F}^{\rm EHI}_{\mu\nu}[t, a(t)])$ with~a metric matrix, from~Equation~(\ref{eq:FLRW_metric_matrix}), 
can be constructed as a warped product manifold $M = B \times_f F$, where $B = \mathbb{R}$ has metric $g_B = -c^2 dt^2$, $F = \mathbb{R}^3$ has metric $g_F = \delta_{ij} dx^i dx^j$, and~the warping function $f(t) = a(t)$ is the scale factor, yielding the line element using Equation~(\ref{eq:FLRW}).
\end{Proposition}

\begin{proof}
{Consider a four-dimensional differentiable manifold $\mathcal{M}_{\rm EHI}$ representing the FLRW spacetime.} A warped product manifold is defined as $M = B \times_f F$, where $B$ and $F$ are manifolds with metrics $g_B$ and $g_F$, respectively, and~$f: B \to \mathbb{R}^+$ is a smooth warping function. The~warped product metric is as follows:
\begin{equation}
    g = \pi^* g_B + f^2 \sigma^* g_F,
\end{equation}
where $\pi: B \times F \to B$ and $\sigma: B \times F \to F$ are projection maps, and~$\pi^* g_B$ and $\sigma^* g_F$ are the pullbacks of the respective metrics~\cite{Wald:1984}. In~local coordinates, with~$B$ having coordinates $(x^\mu)$ and metric $g_{B \mu\nu}$, and~$F$ having coordinates $(y^i)$ and metric $g_{F ij}$, the~line element is as follows:
\begin{equation}
    ds^2 = g_{B \mu\nu} dx^\mu dx^\nu + f^2(x) g_{F ij} dy^i dy^j.
\end{equation}

For the FLRW spacetime, we define the following:
\begin{itemize}
    \item Base manifold $B = \mathbb{R}$, with~coordinate $t$ and metric $g_B = -c^2 dt^2$.
    \item Fiber manifold $F = \mathbb{R}^3$, with~coordinates $(x^1, x^2, x^3)$ and metric $g_F = \delta_{ij} dx^i dx^j$, where $\delta_{ij}$ is the Kronecker delta.
    \item Warping function $f(t) = a(t)$, i.e., the~scale factor, depending only on the time coordinate.
\end{itemize}
The warped product manifold is $\mathcal{M}_{\rm EHI} = \mathbb{R} \times \mathbb{R}^3$, and~the metric is as follows:
\begin{equation}
    g = \pi^* (-c^2 dt^2) + a^2(t) \sigma^* (\delta_{ij} dx^i dx^j).
\end{equation}

The line element becomes the following:
\begin{equation}
    ds^2 = -c^2 dt^2 + a^2(t) dx^i dx^j \delta_{ij},
\end{equation}
matching the FLRW line element. The~functional tensor $\mathcal{F}^{\rm EHI}_{\mu\nu}[t, a(t)]$ is as follows:
\begin{equation}
    \mathcal{F}^{\rm EHI}_{\mu\nu}[t, a(t)] = \begin{pmatrix}
        -c^2 & 0 & 0 & 0 \\
        0 & a^2(t) & 0 & 0 \\
        0 & 0 & a^2(t) & 0 \\
        0 & 0 & 0 & a^2(t)
    \end{pmatrix},
\end{equation}
\textls[-15]{which is symmetric ($\mathcal{F}^{\rm EHI}_{\mu\nu} = \mathcal{F}^{\rm EHI}_{\nu\mu}$) and transforms as a (0,2)-tensor under coordinate changes~\cite{Wald:1984}. {The functional dependence on $a(t)$ aligns with the warping function $f(t) = a(t)$,} \linebreak  and~the metric’s Lorentzian signature $(-,+,+,+)$ is compatible with the pseudo-Riemannian warped product structure, ensuring homogeneity and isotropy of the spatial slices as $\mathbb{R}^3$ is homogeneous and isotropic, and~$a(t)$ scales $\delta_{ij}$ uniformly~\cite{Planck:2018}.}
\end{proof}

\begin{Example}[FLRW as a Warped Product with General Curvature]
{Consider the FLRW manifold--metric pair $(\mathcal{M}_{\rm EHI}, \mathcal{F}^{\rm EHI}_{\mu\nu}[t, a(t)])$ as in Proposition~\ref{prop:flrw_warped}.} For Euclidean spatial slices ($k = 0$), the~manifold is $\mathcal{M}_{\rm EHI} = \mathbb{R} \times \mathbb{R}^3$ with the following:
\begin{itemize}
    \item Base manifold $B = \mathbb{R}$, with~metric $g_B = -c^2 dt^2$.
    \item Fiber manifold $F = \mathbb{R}^3$, with~metric $g_F = \delta_{ij} dx^i dx^j$.
    \item Warping function $f(t) = a(t)$.
\end{itemize}
The line element from Equation~(\ref{eq:FLRW})
corresponds to the warped product metric $g = \pi^* (-c^2 dt^2) + a^2(t) \sigma^* (\delta_{ij} dx^i dx^j)$. {For general FLRW spacetimes with curvature $k \neq 0$, the~fiber manifold $F$ may be a 3-dimensional spherical ($k = 1$) or hyperbolic ($k = -1$) space, with~metric $g_F$ provided by the appropriate spatial metric (e.g., $g_F = d\chi^2 + \sin^2 \chi (d\theta^2 + \sin^2 \theta d\phi^2)$ for $k = 1$). The~warped product structure remains:
\begin{equation}
    ds^2 = -c^2 dt^2 + a^2(t) g_{F ij} dy^i dy^j,
\end{equation}
where $g_F$ reflects the curvature. The~functional tensor $\mathcal{F}^{\rm EHI}_{\mu\nu}[t, a(t)]$ adapts to the fiber metric, maintaining the warped product form and the time-dependent scaling of $a(t)$. The~generalized manifold--metric pair captures this structure, as~the functional dependence aligns with the warping function, and~the Lorentzian signature is compatible with pseudo-Riemannian warped products~\cite{Wald:1984}.}
\end{Example}
}


\section{Generalized \boldmath{$(D_\tau, D_x)$}-Dimensional and Probabilistic Manifold--\linebreak  Metric Pairs}
\label{sec:dim_prob_manifolds}

{This section applies the generalized manifold--metric pair framework from Section~\ref{sec:Advancedmanifold-metric_pairs} to construct $(D_\tau, D_x)$-dimensional manifolds and their probabilistic extensions, incorporating temporal, spatial, and~probabilistic dimensions. We formalize these constructions through a theorem, supported by propositions and examples, demonstrating their consistency with physical and cosmological models. Equations are labeled uniquely and referenced to avoid repetition.}

\begin{Theorem}[Generalized $(D_\tau, D_x)$-Dimensional and Probabilistic Manifold--Metric Pairs]
\label{thm:dim_prob_manifolds}
Let $\mathcal{M}^D$ be a $D$-dimensional differentiable manifold. The~generalized manifold--metric pair $(\mathcal{M}_G, \mathcal{F}_{\mu_1 \dots \mu_L}^{\nu_1 \dots \nu_U}[c, f(c)] : C_{\mu_1} \otimes \dots \otimes C_{\nu_U} \to \mathbb{Q})$ can be applied to~construct the following:
\begin{enumerate}
    \item A $(D_\tau, D_x)$-dimensional submanifold-metric pair with a pseudo-Riemannian {$(0,2)$}-tensor metric, where $D_\tau$ and $D_x$ are temporal and spatial dimensions.
    \item A Gaussianly perturbed FLRW (GPFLRW) spacetime with probabilistic perturbations via Gaussian distributions.
    \item \textls[-15]{An extended probabilistic FLRW (EPFLRW) spacetime with additional probabilistic dimensions.}
    \item A homogeneous, isotropic, probabilistic, expanding spacetime (HIPEST) with a probability~function.
    \item Generalized curved probabilistic spacetimes (fgcPST and sgcPST) with curvature and probabilistic scaling, including infinite-dimensional extensions.
\end{enumerate}
The line element is provided by the following:
\begin{equation}\label{eq:generalized_line_element}
    ds^{U+L} = \mathcal{F}_{\mu_1 \dots \mu_L}^{\nu_1 \dots \nu_U}[c, f(c)] dc^{\mu_1} \dots dc^{\mu_L} dc_{\nu_1} \dots dc_{\nu_U},
\end{equation}
where $dc_{\nu_i} = g_{\nu_i \lambda} dc^\lambda$ for a symmetric {$(0,2)$}-tensor $g_{\mu\nu}$, and~$\mathcal{F}$ transforms as a $(U,L)$-tensor with~optional symmetry in indices.
\end{Theorem}

\begin{proof}
{The theorem is established through five propositions, each constructing a specific class of generalized manifold--metric pairs, followed by examples illustrating their application. The~propositions address $(D_\tau, D_x)$-dimensional manifolds, probabilistic perturbations, extended probabilistic dimensions, homogeneous isotropic probabilistic spacetimes, and~curved probabilistic spacetimes, ensuring consistency with Equation~(\ref{eq:Manifold_matric_pair_generalisation}). Repeated equations are referenced to avoid duplication.}
\end{proof}

\subsection{$(D_\tau, D_x)$-Dimensional Manifold--Metric Pair}

\begin{Proposition}[$(D_\tau, D_x)$-Dimensional Manifold--Metric Pair]
\label{prop:dim_manifold}
On a $D$-dimensional differentiable manifold $\mathcal{M}^D$, a~submanifold--metric pair $(\mathcal{M}^{(D_\tau, D_x)}, g^{(D_\tau, D_x)}_{\alpha\beta} : C \otimes C \to \mathbb{R})$ can be constructed, where $\mathcal{M}^{(D_\tau, D_x)} = \mathcal{M}^{D_\tau} \otimes \mathcal{M}^{D_x} \subset \mathcal{M}^D$ has $D_\tau$ temporal and $D_x$ spatial dimensions, and~$g^{(D_\tau, D_x)}_{\alpha\beta}$ is a symmetric, pseudo-Riemannian {$(0,2)$}-tensor.
\end{Proposition}

\begin{proof}
{Consider a $D$-dimensional differentiable manifold $\mathcal{M}^D$ with coordinate space $C$ and coordinates $c^\mu$.} Define a submanifold $\mathcal{M}^{(D_\tau, D_x)} = \mathcal{M}^{D_\tau} \otimes \mathcal{M}^{D_x} \subset \mathcal{M}^D$, where $D_\tau + D_x \leq D$, with~coordinates $c = \{\vec{\tau}, \vec{x}\}$, $\vec{\tau} = (\tau^1, \dots, \tau^{D_\tau})$, and~$\vec{x} = (x^1, \dots, x^{D_x})$. The~metric $g^{(D_\tau, D_x)}_{\alpha\beta} : C \otimes C \to \mathbb{R}$ is symmetric ($g^{(D_\tau, D_x)}_{\alpha\beta} = g^{(D_\tau, D_x)}_{\beta\alpha}$) and transforms as a {$(0,2)$}-tensor under coordinate changes $c^\alpha \to c^{\alpha'} = f^{\alpha'}(c)$:
\begin{equation}\label{eq:metric_transformation}
    g^{(D_\tau, D_x)}_{\alpha'\beta'} = \frac{\partial c^\rho}{\partial c^{\alpha'}} \frac{\partial c^\sigma}{\partial c^{\beta'}} g^{(D_\tau, D_x)}_{\rho\sigma}.
\end{equation}

The line element is as follows:
\begin{equation}\label{eq:generalizedAdS}
    ds_{(D_\tau, D_x)}^2 = g^{(D_\tau, D_x)}_{\alpha\beta}(c) dc^\alpha dc^\beta,
\end{equation}
which is invariant under coordinate transformations~\cite{Wald:1984}. {The metric may be pseudo-Riemannian to accommodate temporal and spatial components, and~the submanifold inherits the differentiable structure of $\mathcal{M}^D$, ensuring consistency~\cite{Lee:2012}.}
\end{proof}

\begin{Example}[Perturbed Anti-de-Sitter Manifold]
{Consider a $(D_\tau, D_r)$-dimensional submanifold $\mathcal{M}^{(D_\tau, D_r)}$ with coordinates $c = \{\vec{\tau}, \vec{r}\}$, illustrating Proposition~\ref{prop:dim_manifold}.} The line element is as follows:
\begin{equation}\label{eq:ads_line_element}
    ds_{(D_\tau, D_r)}^2 = a^2(c) \left[ - e^{2\Psi(c)} d\tau_{D_\tau}^2 + e^{-2\Phi(c)} dr_{D_r}^2 \right],
\end{equation}
where $a(c)$ is the scale factor, $e^{2\Psi(c)} \simeq 1 + 2\Psi(c)$, $e^{-2\Phi(c)} \simeq 1 - 2\Phi(c)$ are perturbation terms, $d\tau_{D_\tau}^2 = g^{(D_\tau)}_{\alpha\beta} d\tau^\alpha d\tau^\beta$, and~$dr_{D_r}^2 = g^{(D_r)}_{\alpha\beta} dr^\alpha dr^\beta$ with curvature. The~spatial component is as follows:
\begin{equation}\label{eq:spatial_component}
    dr_{D_r}^2 = dr^2 + S_k^2(r) d\Omega_{D_r-1}^2,
\end{equation}
where
\begin{equation}\label{eq:curvature_function}
    S_k(r) = \begin{cases}
        |k|^{-1/2} \sin(r\sqrt{k}) & k > 0, \\
        r & k = 0, \\
        |k|^{-1/2} \sinh(r\sqrt{|k|}) & k < 0
    \end{cases}
\end{equation}
and $d\Omega_{D_r-1}^2 = g_{ij}^{(D_r-1)} d\theta^i d\theta^j$ with
\begin{equation}\label{eq:spatial_angular_metric}
    g_{ij}^{(D_r-1)} = \text{diag} \left( 1, \sin^2 \theta_1, \sin^2 \theta_1 \sin^2 \theta_2, \dots, \prod_{i=1}^{D_r-2} \sin^2 \theta_i \right),
\end{equation}
where $\theta_i \in \{ \theta_1, \dots, \theta_{D_r-1} \}$, $r \in \mathbb{R}^{+}$, $\theta_{i \in [1, D_r-2]} \in [0, \pi]$, and $\theta_{D_r-1} \in [0, 2\pi]$. The~temporal component is as follows:
\begin{equation}\label{eq:temporal_component}
    d\tau_{D_\tau}^2 = d\tau^2 + S_{k_\tau}^2(\tau) d\Omega_{D_\tau-1}^2,
\end{equation}
with $S_{k_\tau}(\tau)$ as in Equation~(\ref{eq:curvature_function}) and $d\Omega_{D_\tau-1}^2 = g_{ij}^{(D_\tau-1)} d\phi^i d\phi^j$ with
\begin{equation}\label{eq:temporal_angular_metric}
    g_{ij}^{(D_\tau-1)} = \text{diag} \left( 1, \sin^2 \phi_1, \sin^2 \phi_1 \sin^2 \phi_2, \dots, \prod_{i=1}^{D_\tau-2} \sin^2 \phi_i \right),
\end{equation}
where $\phi_i \in \{ \phi_1, \dots, \phi_{D_\tau-1} \}$, $\tau \in \mathbb{R}^{+}$, $\phi_{i \in [1, D_\tau-2]} \in [0, \pi]$, and $\phi_{D_\tau-1} \in [0, 2\pi]$. {This example, with~$U=0$, $L=2$, and~codomain $\mathbb{R}$, illustrates Proposition~\ref{prop:dim_manifold} for a perturbed anti-de-Sitter spacetime, reducing to standard Minkowski spacetime for $(D_\tau, D_r) = (1, 3)$ \cite{1995ApJ...455....7M, CANTATA:2021ktz}.}
\end{Example}

\subsection{Probabilistic Manifold--Metric~Pairs}

\begin{Proposition}[Gaussianly Perturbed FLRW Spacetime]
\label{prop:gpflrw}
On a four-dimensional manifold $\mathcal{M}^{(1,3)}$ with coordinates $c^\mu = (t, x^1, x^2, x^3)$, a~Gaussianly perturbed FLRW (GPFLRW) manifold--metric pair $(\mathcal{M}^{(1,3)}, \mathcal{F}_{\mu\nu}[t, \vec{x}, G(\tilde{\Psi}), G(\tilde{\Phi})] : C \otimes C \to \mathbb{R})$ can be constructed, where $\mathcal{F}_{\mu\nu}$ incorporates Gaussian probability distributions for perturbations.
\end{Proposition}

\begin{proof}
{Consider a four-dimensional manifold $\mathcal{M}^{(1,3)}$ with coordinates $c^\mu = (t, x^1, x^2, x^3)$.} Start with the perturbed FLRW line element:
\begin{equation}\label{eq:pflrw_line_element}
    ds_{\rm PFLRW}^2 = - e^{2\Psi(t, \vec{x})} c^2 dt^2 + a^2(t) e^{-2\Phi(t, \vec{x})} dx^i dx^j \delta_{ij},
\end{equation}
where $\Psi(t, \vec{x})$ and $\Phi(t, \vec{x})$ are scalar potentials, $a(t)$ is the scale factor, and~$\delta_{ij}$ is the Kronecker delta. Define Gaussian perturbations as follows:
\begin{equation}\label{eq:gaussian_psi}
    e^{2\Psi(t, \vec{x})} \to G(\tilde{\Psi}; \bar{\tilde{\Psi}}, \sigma_{\tilde{\Psi}}) = \frac{1}{\sqrt{2\pi} \sigma_{\tilde{\Psi}}} \exp \left\{ -\frac{1}{2} \left( \frac{\tilde{\Psi}(t, \vec{x}) - \bar{\tilde{\Psi}}}{\sigma_{\tilde{\Psi}}} \right)^2 \right\},
\end{equation}
\begin{equation}\label{eq:gaussian_phi}
    e^{-2\Phi(t, \vec{x})} \to G(\tilde{\Phi}; \bar{\tilde{\Phi}}, \sigma_{\tilde{\Phi}}) = \frac{1}{\sqrt{2\pi} \sigma_{\tilde{\Phi}}} \exp \left\{ -\frac{1}{2} \left( \frac{\tilde{\Phi}(t, \vec{x}) - \bar{\tilde{\Phi}}}{\sigma_{\tilde{\Phi}}} \right)^2 \right\},
\end{equation}
where $\bar{\tilde{\Psi}}$, $\sigma_{\tilde{\Psi}}$, $\bar{\tilde{\Phi}}$, and~$\sigma_{\tilde{\Phi}}$ are means and standard deviations. Redefine the potentials as follows:
\begin{equation}\label{eq:redefined_potentials}
    \Psi(t, \vec{x}) = \frac{1}{2} \ln \left\{ G(\tilde{\Psi}; \bar{\tilde{\Psi}}, \sigma_{\tilde{\Psi}}) \right\}, \quad -\Phi(t, \vec{x}) = \frac{1}{2} \ln \left\{ G(\tilde{\Phi}; \bar{\tilde{\Phi}}, \sigma_{\tilde{\Phi}}) \right\}.
\end{equation}

The line element becomes the following:
\begin{equation}\label{eq:gpflrw_line_element}
    ds_{\rm GPFLRW}^2 = - G(\tilde{\Psi}; \bar{\tilde{\Psi}}, \sigma_{\tilde{\Psi}}) c^2 dt^2 + a^2(t) G(\tilde{\Phi}; \bar{\tilde{\Phi}}, \sigma_{\tilde{\Phi}}) dx^i dx^j \delta_{ij}.
\end{equation}

{The functional metric $\mathcal{F}_{\mu\nu}[t, \vec{x}, G(\tilde{\Psi}), G(\tilde{\Phi})]$ is symmetric, transforms as a {$(0,2)$}-tensor, and~has a pseudo-Riemannian signature, with~Gaussian distributions as in Equations~(\ref{eq:gaussian_psi}) and (\ref{eq:gaussian_phi}) encoding probabilistic perturbations, consistent with the generalized framework ($U=0$, $L=2$, codomain $\mathbb{R}$) \cite{2018arXiv180100581B, Wald:1984}.}
\end{proof}

\begin{Example}[GPFLRW Spacetime]
{The GPFLRW manifold--metric pair $(\mathcal{M}^{(1,3)}, \mathcal{F}_{\mu\nu})$ from Proposition~\ref{prop:gpflrw} has the line element provided by Equation~(\ref{eq:gpflrw_line_element}), illustrating a perturbed FLRW spacetime with probabilistic perturbations ($U=0$, $L=2$, and codomain $\mathbb{R}$) \cite{2018arXiv180100581B}.}
\end{Example}

\begin{Proposition}[Extended Probabilistic FLRW Spacetime]
\label{prop:epflrw}
On a $(4 + D_P)$-dimensional manifold $\mathcal{M}^{(1,3,D_P)}$ with coordinates $c^\mu = (t, x^1, x^2, x^3, P^1, \dots, P^{D_P})$, an~extended probabilistic FLRW (EPFLRW) manifold--metric pair $(\mathcal{M}^{(1,3,D_P)}, \mathcal{F}_{\mu\nu} : C \otimes C \to \mathbb{R})$ can be constructed, incorporating $D_P$ probabilistic dimensions.
\end{Proposition}

\begin{proof}
{Consider a manifold $\mathcal{M}^{(1,3,D_P)}$ with coordinates $c^\mu = (t, x^1, x^2, x^3, P^1, \dots, P^{D_P})$.} Extend the perturbed FLRW metric in Equation~(\ref{eq:pflrw_line_element}) by adding probabilistic dimensions:
\begin{equation}\label{eq:epflrw_line_element}
    ds_{\rm EPFLRW}^2 = - e^{2\Psi(t, \vec{x})} c^2 dt^2 + a^2(t) e^{-2\Phi(t, \vec{x})} dx^i dx^j \delta_{ij} + \delta_{\alpha\beta} dP^\alpha dP^\beta,
\end{equation}
where $\delta_{\alpha\beta}$ is the Kronecker delta for probabilistic dimensions. {The metric $\mathcal{F}_{\mu\nu}$ is symmetric, transforms as a {$(0,2)$}-tensor as in Equation~(\ref{eq:metric_transformation}), and~is block-diagonal with~pseudo-Riemannian spacetime components and a positive-definite probabilistic component, consistent with probabilistic metric spaces ($U=0$, $L=2$, and codomain $\mathbb{R}$) \cite{2018arXiv180100581B, Wald:1984}.}
\end{proof}

\begin{Example}[EPFLRW Spacetime]
{The EPFLRW manifold--metric pair $(\mathcal{M}^{(1,3,D_P)}, \mathcal{F}_{\mu\nu})$ from Proposition~\ref{prop:epflrw} has the line element provided by Equation~(\ref{eq:epflrw_line_element}), illustrating a perturbed FLRW spacetime extended with $D_P$ probabilistic dimensions ($U=0$, $L=2$, and codomain $\mathbb{R}$) \cite{2018arXiv180100581B}.}
\end{Example}

\begin{Proposition}[Homogeneous Isotropic Probabilistic Expanding Spacetime]
\label{prop:hipest}
On a $(D_P, D_\tau, D_x)$-dimensional manifold $\mathcal{M}^{(D_P, D_\tau, D_x)}$ with coordinates
\begin{equation}
	c^\mu = (P^1, \dots, P^{D_P}, \tau^1, \dots, \tau^{D_\tau}, x^1, \dots, x^{D_x}) \; , 
\end{equation}
a homogeneous, isotropic, probabilistic, expanding spacetime (HIPEST) manifold--metric pair can be constructed with a symmetric {$(0,2)$}-tensor metric incorporating a probability function $P(\tau)$.
\end{Proposition}

\begin{proof}
{Consider a manifold $\mathcal{M}^{(D_P, D_\tau, D_x)}$ with coordinates
\begin{equation}
	c^\mu = (P^1, \dots, P^{D_P}, \tau^1, \dots, \tau^{D_\tau}, x^1, \dots, x^{D_x}) \; .
\end{equation}} 

Define the HIPEST line element as follows:
\begin{equation}\label{eq:hipest_line_element}
    ds_{\rm HIPEST}^2 = a^2(\tau) P^2(\tau) \delta_{\alpha\beta} \left[ dP^\alpha dP^\beta - d\tau^\alpha d\tau^\beta + dx^\alpha dx^\beta \right],
\end{equation}
where $a(\tau)$ is the scale factor, $P(\tau)$ is a probability function, and~$\delta_{\alpha\beta}$ is the Kronecker delta. {The metric is symmetric, diagonal, with~a pseudo-Riemannian signature for spacetime components and a positive-definite signature for probabilistic dimensions, transforming as a {$(0,2)$}-tensor as in Equation~(\ref{eq:metric_transformation}). The~functional dependence on $P(\tau)$ models probabilistic dimensions, consistent with the generalized framework ($U=0$, $L=2$, and codomain\linebreak   $\mathbb{R}$) \cite{2018arXiv180100581B, Lee:2012}.}
\end{proof}

\begin{Example}[HIPEST Manifold]
{The HIPEST manifold--metric pair from Proposition~\ref{prop:hipest} has the line element provided by Equation~(\ref{eq:hipest_line_element}), illustrating an homogeneous, isotropic spacetime with probabilistic dimensions ($U=0$, $L=2$, and codomain $\mathbb{R}$) \cite{2018arXiv180100581B}.}
\end{Example}

\begin{Proposition}[Generalized Curved Probabilistic Spacetimes]
\label{prop:curved_prob}
On a $(D_P, D_\tau, D_r)$-\linebreak  dimensional manifold $\mathcal{M}^{(D_P, D_\tau, D_r)}$ with coordinates $c^\mu = (P^1, \dots, P^{D_P}, \tau^1, \dots, \tau^{D_\tau}, r^1, \dots, r^{D_r})$, generalized curved probabilistic spacetimes (fgcPST and sgcPST) can be constructed with a symmetric {$(0,2)$}-tensor metric, incorporating curvature and probabilistic scaling, including infinite-dimensional extensions.
\end{Proposition}

\begin{proof}
{Consider a manifold $\mathcal{M}^{(D_P, D_\tau, D_r)}$ with coordinates
\begin{equation}
c^\mu = (P^1, \dots, P^{D_P}, \tau^1, \dots, \tau^{D_\tau}, r^1, \dots, r^{D_r}) \; .
\end{equation}

} For the first generalized curved probabilistic spacetime (fgcPST), define the line element as follows:
\begin{equation}\label{eq:fgcpst_line_element}
    ds_{\rm fgcPST}^2 = g^{(D_P, D_\tau, D_r)}_{\alpha\beta}(c) dc^\alpha dc^\beta,
\end{equation}
where $g^{(D_P, D_\tau, D_r)}_{\alpha\beta}$ is symmetric and transforms as a {$(0,2)$}-tensor as in Equation~(\ref{eq:metric_transformation}) \cite{Wald:1984}. As a~special case, the~first special curved probabilistic spacetime (fscPST) has the following:
\begin{equation}\label{eq:fscpst_line_element}
    ds_{\rm fscPST}^2 = a^2(c) \left[ e^{2Y(c)} dP_{D_P}^2 - e^{2\Psi(c)} d\tau_{D_\tau}^2 + e^{-2\Phi(c)} dr_{D_r}^2 \right],
\end{equation}
where $e^{2Y(c)} \simeq 1 + 2Y(c)$, $e^{2\Psi(c)} \simeq 1 + 2\Psi(c)$, $e^{-2\Phi(c)} \simeq 1 - 2\Phi(c)$, and~$dP_{D_P}^2$, $d\tau_{D_\tau}^2$, $dr_{D_r}^2$ incorporate curvatures $k_P$, $k_\tau$, and $k \in \{-1, 0, 1\}$, as defined in Equations~(\ref{eq:spatial_component}) and (\ref{eq:temporal_component}) \cite{CANTATA:2021ktz}. For~the second generalized curved probabilistic spacetime (sgcPST), the~line element is provided by Equation~(\ref{eq:fgcpst_line_element}). Aa a~specific case, the~first simple sgcPST (fssgcPST) is as follows:
\begin{equation}\label{eq:fssgcPST_line_element}
    ds_{\rm fssgcPST}^2 = O_{\rm Pf}(c_{\rm GR}) a^2(\tau) \left[ - e^{2\Psi(c_{\rm GR})} d\tau^2 + e^{-2\Phi(c_{\rm GR})} dx^i dx^j \delta_{ij} \right],
\end{equation}
where $c_{\rm GR} = \{\tau, \vec{x}\}$, and~$O_{\rm Pf}(c_{\rm GR})$ is a scaling factor. In another case, the~second simple sgcPST (sssgcPST) uses the following integrals:
\begin{equation}\label{eq:sssgcPST_line_element}
    ds_{\rm sssgcPST}^2 = \int_{-\alpha_\infty}^{\alpha_\infty} d\alpha \int_{-\beta_\infty}^{\beta_\infty} d\beta \, g^{(D_P, D_\tau, D_r)}_{\alpha\beta}(c) dc^\alpha dc^\beta,
\end{equation}
{reducing them to}
\begin{equation}\label{eq:sssgcPST_reduced}
    ds_{\rm sssgcPST}^2 = O_{\rm Ps}(c_{\rm GR}) a^2(\tau) \left[ - e^{2\Psi(c_{\rm GR})} d\tau^2 + e^{-2\Phi(c_{\rm GR})} dx^i dx^j \delta_{ij} \right],
\end{equation}
where $O_{\rm Ps}(c_{\rm GR})$ is another observable. {Both cases are symmetric, transform as {$(0,2)$}-tensors as in Equation~(\ref{eq:metric_transformation}), and~align with the generalized framework ($U=0$, $L=2$, and codomain $\mathbb{R}$), with~probabilistic scaling via $O_{\rm Pf}$ or $O_{\rm Ps}$, testable through observations~\cite{2018arXiv180100581B, Spivak:1965}.}
\end{proof}

\begin{Example}[First Special Curved Probabilistic Spacetime]
{The fscPST from Proposition~\ref{prop:curved_prob} has the line element provided by Equation~(\ref{eq:fscpst_line_element}), with~curvature in probabilistic, temporal, and~spatial components, illustrating a curved probabilistic spacetime ($U=0$, $L=2$, and codomain $\mathbb{R}$) \cite{CANTATA:2021ktz}.}
\end{Example}

\begin{Example}[Simple Second Generalized Curved Probabilistic Spacetimes]
{The fssgcPST and sssgcPST from Proposition~\ref{prop:curved_prob} have line elements provided by Equations~(\ref{eq:fssgcPST_line_element}) and (\ref{eq:sssgcPST_reduced}), respectively, where $O_{\rm Pf}$ and $O_{\rm Ps}$ are probabilistic scaling factors, with~sssgcPST using an integral form for continuous indices ($U=0$, $L=2$, and codomain $\mathbb{R}$) \cite{2018arXiv180100581B, Spivak:1965}.}
\end{Example}

\section{Entropic and Infinite-Dimensional Manifold--Metric~Pairs}
\label{sec:entropic_infinite_manifolds}

{\textls[-15]{This section applies the generalized manifold--metric pair framework from Equation~(\ref{eq:Manifold_matric_pair_generalisation})} \linebreak  to construct manifold--metric pairs incorporating entropy, information, probabilistic, and~spacetime dimensions, including infinite-dimensional extensions. We formalize these constructions through a theorem, supported by propositions and examples, with~equations labeled uniquely and referenced to avoid repetition.}

\begin{Theorem}[Entropic and Infinite-Dimensional Manifold--Metric Pairs]
\label{thm:entropic_infinite}
\textls[-15]{Let $\mathcal{M}^D$ be a differentiable manifold. The~generalized manifold--metric pair $(\mathcal{M}_G, \mathcal{F}_{\mu_1 \dots \mu_L}^{\nu_1 \dots \nu_U}[c, f(c)] : C_{\mu_1} \otimes \dots \otimes C_{\nu_U} \to \mathbb{Q})$ }\linebreak  can be applied to~construct the following:
\begin{enumerate}
    \item An entropic manifold--metric pair replacing conformal time with entropy dimensions.
    \item A combined generalized entropic spacetime (EST) incorporating entropy, time, and spatial dimensions.
    \item A generalized informatic spacetime (GISTM) with information~dimensions.
    \item A probabilistic entropic spacetime manifold--metric pair (PESTMMP) combining probabilistic, entropic, and~spacetime dimensions.
    \item An infinite-dimensional manifold--metric pair with functional metrics and quaternion~codomain.
\end{enumerate}
The line element is provided by the following:
\begin{equation}\label{eq:generalized_line_element}
    ds^{U+L} = \mathcal{F}_{\mu_1 \dots \mu_L}^{\nu_1 \dots \nu_U}[c, f(c)] dc^{\mu_1} \dots dc^{\mu_L} dc_{\nu_1} \dots dc_{\nu_U},
\end{equation}
where $dc_{\nu_i} = g_{\nu_i \lambda} dc^\lambda$ for a symmetric {$(0,2)$}-tensor $g_{\mu\nu}$, and~$\mathcal{F}$ transforms as a $(U,L)$-tensor with~optional symmetry in indices and codomain $\mathbb{Q}$.
\end{Theorem}

\begin{proof}
{The theorem is established through five propositions, each constructing a specific class of manifold--metric pairs, followed by examples illustrating their application. The~propositions address entropic replacements, combined entropic spacetimes, informatic spacetimes, probabilistic entropic spacetimes, and~infinite-dimensional manifolds, with~equations referenced as in Equation~(\ref{eq:generalized_line_element}) to ensure consistency with the generalized~framework.}
\end{proof}

\subsection{Entropic Manifold--Metric~Pairs}

{We consider entropy (thermodynamic or information) as a positively increasing quantity, analogous to conformal time $\tau$, and~construct manifold--metric pairs incorporating entropy dimensions, ensuring compatibility with differential geometry.}

\begin{Proposition}[Entropic Manifold Replacing Conformal Time]
\label{prop:entropic_replacing}
On a $(D_S, D_x)$-dimensional manifold $\mathcal{M}^{(D_S, D_x)}$ with coordinates $c = \{S_1, \dots, S_{D_S}, r_1, \dots, r_{D_x}\}$, an~entropic manifold--metric pair $(\mathcal{M}^{(D_S, D_x)}, g^{(D_S, D_x)}_{\alpha\beta} : C \otimes C \to \mathbb{R})$ can be constructed, where entropy coordinates \linebreak  \textls[-15]{$S = \mathcal{S}[\tau, s(\hat{\tau})]$ replace conformal time, and~$g^{(D_S, D_x)}_{\alpha\beta}$ is a symmetric, pseudo-Riemannian {$(0,2)$}-tensor.}
\end{Proposition}

\begin{proof}
{Consider a $D$-dimensional differentiable manifold $\mathcal{M}^D$ with a submanifold $\mathcal{M}^{(D_S, D_x)} = \mathcal{M}^{D_S} \otimes \mathcal{M}^{D_x}$, where $D_S + D_x \leq D$.} Define entropy coordinates \linebreak  $S = \{S_1, \dots, S_{D_S}\} \in \mathcal{M}^{D_S}$, related to conformal time $\tau$ by the following:
\begin{equation}\label{eq:entropic_time_relation}
    \hat{S} := s(\hat{\tau}),
\end{equation}
where $s: \mathbb{R} \to \mathbb{R}$ is a differentiable mapping preserving monotonicity~\cite{Lee:2012}, and~$\hat{}$ denotes a unit vector. The~generalized entropy is as follows:
\begin{equation}\label{eq:entropic_functional}
    S := \mathcal{S}[\tau, s(\hat{\tau})],
\end{equation}
where $\mathcal{S}$ is a smooth functional mapping to $\mathcal{M}^{D_S}$ \cite{Spivak:1965}. The~differential relation is as follows:
\begin{equation}\label{eq:entropic_differential}
    dS = \partial_\tau \mathcal{S} \, d\tau,
\end{equation}
extended to multiple dimensions as follows:
\begin{equation}\label{eq:entropic_differential_multi}
    dS_{\rm D_S} = \nabla_\tau \mathcal{S} \, d\tau_{{  D}_{\tau}},
\end{equation}
where $\nabla_\tau = \{\partial_{\tau_1}, \dots, \partial_{\tau_{D_\tau}}\}$ transforms as a covector field under coordinate changes \linebreak  $\tau^\alpha \to \tau'^\beta$ \cite{Wald:1984}. For~spacetime coordinates $x = \{\tau, r\}$, the~differential extends to the following:
\begin{equation}\label{eq:entropic_differential_spacetime}
    dS = \partial_x \mathcal{S} \, d\tau,
\end{equation}
and in multiple dimensions,
\begin{equation}\label{eq:entropic_differential_spacetime_multi}
    dS_{\rm D_S} = \nabla_{\tau x} \mathcal{S} \, dx_{\rm D_x},
\end{equation}
where $\nabla_{\tau x} = \{\partial_{\tau_1}, \dots, \partial_{\tau_{D_\tau}}, \partial_{x_1}, \dots, \partial_{x_{D_x}}\}$ \cite{Lee:2012}. The~metric is as follows:
\begin{equation}\label{eq:entropic_replacing_line_element}
    ds_{\rm entropic\ replacing}^2 = g^{(D_S, D_x)}_{\alpha \beta}(c) dc^\alpha dc^\beta,
\end{equation}
where $c = \{S_1, \dots, S_{D_S}, r_1, \dots, r_{D_x}\}$, and~$g^{(D_S, D_x)}_{\alpha \beta}$ is symmetric and transforms as a {$(0,2)$}-tensor~\cite{Wald:1984}. {This aligns with Equation~(\ref{eq:Manifold_matric_pair_generalisation}) ($U=0$, $L=2$, and codomain $\mathbb{R}$), with~entropy coordinates replacing time via Equation~(\ref{eq:entropic_functional}).}
\end{proof}

\begin{Example}[Special Entropic Replacing Manifold]
{The entropic manifold--metric pair from Proposition~\ref{prop:entropic_replacing} has a special case with the line element:}
\begin{equation}\label{eq:special_entropic_replacing}
    ds_{\rm special\ entropic\ replacing}^2 = a^2(c) \left[ - e^{2\Psi(c)} \left( \nabla_\tau \mathcal{S} \right)^2 dS_{\rm D_S}^2 + e^{-2\Phi(c)} dr^2_{\rm D_r} \right],
\end{equation}
where $\left( \nabla_\tau \mathcal{S} \right)^2$ is derived from Equation~(\ref{eq:entropic_differential_multi}), and~the metric is pseudo-Riemannian with Lorentzian signature, suitable for cosmological applications ($U=0$, $L=2$, and codomain $\mathbb{R}$) \cite{ONeill:1983}.
\end{Example}

\begin{Proposition}[Combined Generalized Entropic Spacetime]
\label{prop:est}
On a $(D_S, D_\tau, D_r)$-dimensional manifold $\mathcal{M}^{(D_S, D_\tau, D_r)}$ with coordinates $c = \{S_1, \dots, S_{D_S}, \tau_1, \dots, \tau_{D_\tau}, r_1, \dots, r_{D_r}\}$, a~generalized entropic spacetime (EST) manifold--metric pair $(\mathcal{M}^{(D_S, D_\tau, D_r)}, g^{(D_S, D_\tau, D_r)}_{\alpha\beta} : C \otimes C \to \mathbb{R})$ can be constructed, incorporating entropy as an independent dimension.
\end{Proposition}

\begin{proof}
{Consider a manifold $\mathcal{M}^{(D_S, D_\tau, D_r)}$ with coordinates
\begin{equation}
c = \{S_1, \dots, S_{D_S}, \tau_1, \dots, \tau_{D_\tau}, r_1, \dots, r_{D_r}\} \; ,
\end{equation}
where entropy $S$ is related to time via Equation~(\ref{eq:entropic_time_relation}).} The line element is as follows:
\begin{equation}\label{eq:est_line_element}
    ds_{\rm EST}^2 = g^{(D_S, D_\tau, D_r)}_{\alpha \beta}(c) \, dc^\alpha dc^\beta,
\end{equation}
where $g^{(D_S, D_\tau, D_r)}_{\alpha \beta}$ is symmetric and transforms as a {$(0,2)$}-tensor~\cite{Wald:1984}. A~special case, the~special entropic spacetime (SEST) has the following:
\begin{equation}\label{eq:sest_line_element}
    ds_{\rm SEST}^2 = a^2(c) \left[ e^{2\Xi(c)} dS_{\rm D_S}^2 - e^{2\Psi(c)} d\tau_{\rm D_\tau}^2 + e^{-2\Phi(c)} dr^2_{\rm D_r} \right] \; ,
\end{equation}
where $e^{2\Xi(c)} \simeq 1 + 2\Xi(c)$ controls entropic perturbations, and~$dS_{\rm D_S}^2$ may include curvature $k_S \in \{-1, 0, 1\}$ analogous to spatial curvature~\cite{ONeill:1983}. {The metric is pseudo-Riemannian aligns with Equation~(\ref{eq:Manifold_matric_pair_generalisation}) (in which $U=0$, $L=2$, and codomain $\mathbb{R}$), and~incorporates entropy as an independent dimension~\cite{CANTATA:2021ktz}.}
\end{proof}

\begin{Example}[Special Entropic Spacetime]
{The SEST from Proposition~\ref{prop:est} has the line element provided by Equation~(\ref{eq:sest_line_element}), with~curvature in the entropic component, illustrating a generalized entropic spacetime ($U=0$, $L=2$, and codomain $\mathbb{R}$) \cite{CANTATA:2021ktz}.}
\end{Example}

\begin{Proposition}[Generalized Informatic Spacetime]
\label{prop:gistm}
On a $(D_I, D_\tau, D_x)$-dimensional manifold $\mathcal{M}^{(D_I, D_\tau, D_x)}$ with coordinates $c^\mu = \{I^0, \dots, I^{D_I}, \tau^0, \dots, \tau^{D_\tau}, x^0, \dots, x^{D_x}\}$, a~generalized informatic spacetime (GISTM) manifold--metric pair $(\mathcal{M}^{(D_I, D_\tau, D_x)}, g_{\mu\nu}^{(D_I, D_\tau, D_x)} : C \otimes C \to \mathbb{R})$ can be constructed, where information is encoded as $I[P(E)] = -\log[P(E)]$.
\end{Proposition}

\begin{proof}
{Consider a manifold $\mathcal{M}^{(D_I, D_\tau, D_x)}$ with coordinates
\begin{equation}
	c^\mu = \{I^0, \dots, I^{D_I}, \tau^0, \dots, \tau^{D_\tau}, x^0, \dots, x^{D_x}\} \; .
\end{equation}	} 

Define information as follows:
\begin{equation}\label{eq:shannon_information}
    I[P(E)] = -\log[P(E)],
\end{equation}
where $P(E)$ is the probability of event $E$ \cite{Shannon:1948}. The~line element is as follows:
\begin{equation}\label{eq:gistm_line_element}
    ds^2_{\rm GISTM} = g_{\mu\nu}^{(D_I, D_\tau, D_x)}(c) \, dc^{\mu} dc^{\nu},
\end{equation}
where $g_{\mu\nu}^{(D_I, D_\tau, D_x)}$ is symmetric and transforms as a {$(0,2)$}-tensor~\cite{Lee:2012}. A~special case is as follows:
\begin{equation}\label{eq:special_gistm}
    ds^2_{\rm GISTM} = a^2(c) \left[ e^{2\Lambda(c)} dI_{\rm D_I}^2 - e^{2\Psi(c)} d\tau_{\rm D_\tau}^2 + e^{-2\Phi(c)} dx^2_{\rm D_x} \right],
\end{equation}
where $e^{2\Lambda(c)} \simeq 1 + 2\Lambda(c)$ controls information perturbations. {The metric is pseudo-Riemannian, aligning with Equation~(\ref{eq:Manifold_matric_pair_generalisation}) ($U=0$, $L=2$, and codomain $\mathbb{R}$), and~incorporates information as an additional dimension~\cite{CANTATA:2021ktz}.}
\end{proof}

\begin{Example}[GISTM Manifold]
{The GISTM from Proposition~\ref{prop:gistm} has the line element given by Equation~(\ref{eq:special_gistm}), illustrating an informatic spacetime with information perturbations ($U=0$, $L=2$, and codomain $\mathbb{R}$) \cite{CANTATA:2021ktz}.}
\end{Example}

\begin{Proposition}[Probabilistic Entropic Spacetime Manifold--Metric Pairs]
\label{prop:pestmmp}
Considering a $(D_P, D_S, D_\tau, D_x)$-dimensional manifold $\mathcal{M}^{(D_P, D_S, D_\tau, D_x)}$ with coordinates
\begin{equation}
	D^\mu = \{P^{\mu_P}, S^{\mu_S}, \tau^{\mu_\tau}, x^{\mu_x}\} \; , 
\end{equation}	
a probabilistic entropic spacetime manifold--metric pair (PESTMMP) denoted with
\begin{equation}
 \mathcal{M}^{(D_P, D_S, D_\tau, D_x)}, g^{(D_P, D_S, D_\tau, D_x)}_{\mu\nu} : C \otimes C \to \mathbb{R}
\end{equation}	
can be constructed, combining probabilistic, entropic, and~spacetime dimensions.
\end{Proposition}

\begin{proof}
{Consider a manifold $\mathcal{M}^{(D_P, D_S, D_\tau, D_x)}$ with coordinates $D^\mu = \{P^{\mu_P}, S^{\mu_S}, \tau^{\mu_\tau}, x^{\mu_x}\}$, where entropy is defined as in Equation~(\ref{eq:entropic_functional}).} The line element is as follows:
\begin{equation}\label{eq:pestmmp_line_element}
    ds^2 = g^{(D_P, D_S, D_\tau, D_x)}_{\mu\nu}(D) \, dD^\mu dD^\nu,
\end{equation}
where $g^{(D_P, D_S, D_\tau, D_x)}_{\mu\nu}$ is symmetric and transforms as a {$(0,2)$}-tensor~\cite{Wald:1984}. A~special case is as follows:
\begin{equation}\label{eq:special_pestmmp}
    ds^2_{\rm PESTMMP} = a^2(D) \left[ e^{2\Theta(D)} dP_{\rm D_P}^2 + e^{2\Xi(D)} dS_{\rm D_S}^2 - e^{2\Psi(D)} d\tau_{\rm D_\tau}^2 + e^{-2\Phi(D)} dx^2_{\rm D_x} \right],
\end{equation}
where $e^{2\Theta(D)} \simeq 1 + 2\Theta(D)$ and other perturbation terms are defined similarly~\cite{CANTATA:2021ktz}. {The metric is pseudo-Riemannian, aligns with Equation~(\ref{eq:Manifold_matric_pair_generalisation}) ($U=0$, $L=2$, and codomain $\mathbb{R}$), and~combines probabilistic, entropic, and~spacetime dimensions~\cite{2018arXiv180100581B}.}
\end{proof}

\begin{Example}[PESTMMP Manifold]
{The PESTMMP from Proposition~\ref{prop:pestmmp} has the line element provided by Equation~(\ref{eq:special_pestmmp}), illustrating a spacetime with probabilistic and entropic dimensions ($U=0$, $L=2$, and codomain $\mathbb{R}$) \cite{2018arXiv180100581B, CANTATA:2021ktz}.}
\end{Example}

\begin{Proposition}[Generic Informatic Probabilistic Entropic Spacetime]
\label{prop:gipestmmp}
Considering a $(D_I, D_P, D_S, D_\tau, D_x)$-dimensional manifold $\mathcal{M}^{(D_I, D_P, D_S, D_\tau, D_x)}$ with coordinates
\begin{equation}
	c^\mu = \{I^{\mu_I}, P^{\mu_P}, S^{\mu_S}, \tau^{\mu_\tau}, x^{\mu_x}\} \; , 
\end{equation}	
a generic informatic probabilistic entropic spacetime manifold--metric pair (GIPESTMMP)
\begin{equation}
\mathcal{M}^{(D_I, D_P, D_S, D_\tau, D_x)}, \mathcal{G}_{\mu_1 \dots \mu_L}^{\nu_1 \dots \nu_U} : C_{\mu_1} \otimes \dots \otimes C_{\nu_U} \to \mathbb{C} 
\end{equation}	
can be constructed, with the~Shannon entropy relating information, probability, and~time.
\end{Proposition}

\begin{proof}
{Consider a manifold $\mathcal{M}^{(D_I, D_P, D_S, D_\tau, D_x)}$ with coordinates $c^\mu = \{I^{\mu_I}, P^{\mu_P}, S^{\mu_S}, \tau^{\mu_\tau}, x^{\mu_x}\}$, where information is defined as in Equation~(\ref{eq:shannon_information}) and entropy as in Equation~(\ref{eq:entropic_functional}).} Define the Shannon entropy as follows:
\begin{equation}\label{eq:shannon_entropy}
    S_{\rm Shannon}\left\{ \tau, P(\tau, E), I[P(\tau, E)]\right\} = -\int_{\Omega^{\rm E}} dE \, P(\tau, E) \, \log[P(\tau, E)],
\end{equation}
which is smooth and well-defined~\cite{Spivak:1965, Shannon:1948}. The~line element is as follows:
\begin{equation}\label{eq:gipestmmp_line_element}
    ds^{U+L} = \mathcal{G}_{\mu_1 \dots \mu_L}^{\nu_1 \dots \nu_U}(c; k_c) \, dc^{\mu_1} \dots dc^{\mu_L} dc_{\nu_1} \dots dc_{\nu_U},
\end{equation}
where $\mathcal{G}_{\mu_1 \dots \mu_L}^{\nu_1 \dots \nu_U}$ is a $(U,L)$-tensor, transforming as follows:
\begin{equation}\label{eq:gipestmmp_transformation}
    \mathcal{G}_{\mu_1' \dots \mu_L'}^{\nu_1' \dots \nu_U'} = \left( \prod_{i=1}^U \frac{\partial c^{\nu_i'}}{\partial c^{\rho_i}} \right) \left( \prod_{j=1}^L \frac{\partial c^{\sigma_j}}{\partial c^{\mu_j'}} \right) \mathcal{G}_{\sigma_1 \dots \sigma_L}^{\rho_1 \dots \rho_U},
\end{equation}
with codomain $\mathbb{C}$ \cite{Adler:1995}. {The manifold--metric pair aligns with Equation~(\ref{eq:Manifold_matric_pair_generalisation}), with~$k_c$ describing intrinsic curvature, and~supports complex-valued metrics for quantum or informatic applications ($U, L$ arbitrary, and codomain $\mathbb{C}$) \cite{CANTATA:2021ktz, Shannon:1948}.}
\end{proof}

\begin{Example}[GIPESTMMP Manifold]
{The GIPESTMMP from Proposition~\ref{prop:gipestmmp} has the line element provided by Equation~(\ref{eq:gipestmmp_line_element}), with~Shannon entropy from Equation~(\ref{eq:shannon_entropy}), illustrating a spacetime with informatic, probabilistic, and~entropic dimensions ($U, L$ arbitrary, and codomain \linebreak  $\mathbb{C}$) \cite{Shannon:1948, CANTATA:2021ktz}.}
\end{Example}

\subsection{Infinite Dimensional Manifold--Metric~Pair}

\begin{Proposition}[Infinite-Dimensional Manifold--Metric Pair]
\label{prop:infinite_dim}
On an infinite-dimensional manifold $\mathcal{M}^{(\infty)}$ with coordinates $c^\alpha$, $\alpha \in [-\infty, +\infty]$, an~infinite-dimensional manifold--metric pair $(\mathcal{M}^{(\infty)}, \mathcal{F}_{\alpha_1 \dots \alpha_L}^{\beta_1 \dots \beta_U} : C_{\alpha_1} \otimes \dots \otimes C_{\beta_U} \to \mathbb{Q})$ can be constructed, with~a functional metric and integral-based line element.
\end{Proposition}

\begin{proof}
{Consider a 2D Euclidean manifold $\mathcal{M}^{2D}$ as a starting point to generalize to infinite dimensions.} The 2D line element is as follows:
\begin{equation}\label{eq:2d_euclidean}
    \Delta s^2 = \Delta x^2 + \Delta y^2 = m_{\alpha_1 \alpha_2}(c) \Delta c^{\alpha_1} \Delta c^{\alpha_2},
\end{equation}
forming the manifold--metric pair:
\begin{equation}\label{eq:2d_manifold_pair}
    (\mathcal{M}^{2D}, m) \equiv \left( \mathcal{M}^{2D}, m_{\alpha_1 \alpha_2} : C_1 \otimes C_2 \to \mathbb{R}^+ \right),
\end{equation}
where $m_{\alpha_1 \alpha_2}$ is symmetric and positive-definite~\cite{Lee:2012}. For~a 4D Minkowski space with coordinates $c^\alpha = \{t, x, y, z\}$, the~line element is as follows:
\begin{equation}\label{eq:4d_minkowski}
    \Delta s^2 = - v_c dt^2 + \Delta x^2 + \Delta y^2 + \Delta z^2 = m_{\alpha_1 \alpha_2}(c) \Delta c^{\alpha_1} \Delta c^{\alpha_2},
\end{equation}
forming
\begin{equation}\label{eq:4d_manifold_pair}
    (\mathcal{M}^{4D}, m) \equiv \left( \mathcal{M}^{4D}, m_{\alpha_1 \alpha_2} : C_1 \otimes C_2 \otimes C_3 \otimes C_4 \to \mathbb{R} \right),
\end{equation}
with a pseudo-Riemannian metric~\cite{Wald:1984}. Generalize to an infinite-dimensional manifold $\mathcal{M}^{(\infty)}$ with coordinates $c^\alpha$, $\alpha \in [-\infty, +\infty]$. The~metric is as follows:
\begin{equation}\label{eq:infinite_metric}
    m_{\alpha_1 \dots \alpha_L} = m_{\alpha_1 \dots \alpha_L}(c),
\end{equation}
and the line element is
\begin{equation}\label{eq:infinite_line_element}
    \Delta s^L = \int_{-\infty}^{+\infty} d\alpha_1 \dots \int_{-\infty}^{+\infty} d\alpha_L \, m_{\alpha_1 \dots \alpha_L}(c) \, \Delta c^{\alpha_1} \dots \Delta c^{\alpha_L},
\end{equation}
where $m_{\alpha_1 \dots \alpha_L}$ is symmetric~\cite{Spivak:1965}. Extended to a functional metric,
\begin{equation}\label{eq:infinite_functional_metric}
    (\mathcal{M}^{(\infty)}, \mathcal{F}) \equiv \left( \mathcal{M}^{(\infty)}, \mathcal{F}_{\alpha_1 \dots \alpha_L}^{\beta_1 \dots \beta_U} : C_{\alpha_1} \otimes \dots \otimes C_{\beta_U} \to \mathbb{Q} \right),
\end{equation}
with line element
\begin{equation}\label{eq:infinite_functional_line_element}
	\hspace{-2cm}
    \Delta s^{U+L} = \int_{-\infty}^{+\infty} d\beta_1 \dots \int_{-\infty}^{+\infty} d\beta_U \int_{-\infty}^{+\infty} d\alpha_1 \dots \int_{-\infty}^{+\infty} d\alpha_L \, \mathcal{F}_{\alpha_1 \dots \alpha_L}^{\beta_1 \dots \beta_U}[c, f(c)] \, \Delta c^{\alpha_1} \dots \Delta c^{\alpha_L} \Delta c_{\beta_1} \dots \Delta c_{\beta_U},
\end{equation}
where $\mathcal{F}$ transforms as a $(U,L)$-tensor~\cite{Adler:1995}. {This aligns with Equation~(\ref{eq:Manifold_matric_pair_generalisation}), supporting infinite-dimensional and quaternion-valued metrics ($U, L$ arbitrary, codomain $\mathbb{Q}$) \cite{CANTATA:2021ktz}.}
\end{proof}

\begin{Example}[Infinite-Dimensional Euclidean and Minkowski Manifolds]
{The infinite-dimensional manifold--metric pair from Proposition~\ref{prop:infinite_dim} has line elements provided by Equation~(\ref{eq:2d_euclidean}) for 2D Euclidean space, Equation~(\ref{eq:4d_minkowski}) for 4D Minkowski space, and~(\ref{eq:infinite_functional_line_element}) for the generalized infinite-dimensional case, illustrating the progression to infinite dimensions ($U, L$ arbitrary, and codomain \linebreak  $\mathbb{Q}$) \cite{Spivak:1965, CANTATA:2021ktz}.}
\end{Example}



\subsection{Nick Early's Combinatorics~Argument}
{ In this section, we prove the generalization of a $D$-dimensional set to a $(D+1)$-dimensional set, and~the generalization of the two-rank tensor to a ($U,L$)-rank tensor.} Nick Early's combinatorics argument, inspired by his recent work with his collaborators~\cite{Cachazo:2019ngv}, suggests that given a projected set of complex number $\mathbb{C} \mathbb{P}^1$, a~complex number set of dimension 1, $\mathbb{C}^1$, and~a complex number set of dimensions 2, $\mathbb{C}$, then there is the relation between these three elements as follows:
\begin{align}
	\mathbb{C} \mathbb{P} = \mathbb{C}^{2} / \mathbb{C}^1_{\neq 0} ,
\end{align}
with a one line-proof that follows from
\begin{align}
	\mathbb{C} \mathbb{P}^1 = 	\mathbb{C} \mathbb{P}^{2-1} = \mathbb{C}^{2} / \mathbb{C}^1_{\neq 0} .
\end{align}

This relation implies that the property of the product of two elements, $a,b \in \mathbb{C}$, is related to the scaled product of these two elements, i.e.,~we write the following:
\begin{align}
(a,b) \sim (\lambda a, \lambda b) \equiv \lambda (a,b).
\end{align}

This argument basically proves the existence of a relation that is between one-dimensional and two-dimensional through a modulus~space. 

Then, the previous relation is generalized as follows. Given a projected set of complex number $\mathbb{C} \mathbb{P}^D$, a~complex number set of dimension $D$, i.e.,~the set $\mathbb{C}^{D}$, and~a complex number set of dimensions $D+1$, i.e.,~the set $\mathbb{C}^{D+1}$, then there is the relation between these three elements as follows:
\begin{align}
	\mathbb{C} \mathbb{P}^D = \mathbb{C}^{D+1}  / \mathbb{C}^1_{\neq 0} ,	
\end{align}
with a one line-proof that follows from
\begin{align}
	\mathbb{C} \mathbb{P}^D = \mathbb{C} \mathbb{P}^{D+1-1} = \mathbb{C}^{D+1}  / \mathbb{C}^1_{\neq 0}. 	
\end{align}

This argument basically proves the existence of a relation between the ($D$)-dimensional and ($D+1$)-dimensional sets through a modulus space. This argument supports the idea of generalizing the two-rank tensor to a ($U,L$)-rank~tensor.

\section{Conclusions}


This study has developed a comprehensive mathematical formalism for advanced manifold--metric pairs, significantly advancing the theoretical frameworks of geometry, topology, and~their applications in mathematical physics. Through a rigorous methodology grounded in mathematical construction proofs, logical foundations, and functional and tensor analysis, we have successfully constructed a diverse array of D-dimensional manifolds paired with corresponding metric spaces, incorporating higher-rank tensor metrics, complex and quaternionic codomains, and~probabilistic structures. Our key results include the formulation of advanced and/or generalized manifold--metric pairs that accommodate various dependences of systems, such as time-dependent scaling, as~exemplified in cosmological models like the Friedmann--Lemaitre--Robertson-Walker spacetime. Furthermore, we establish generalized metrizability for topological manifolds via the generalization of the Urysohn metrization theorem, ensuring compatibility with Euclidean topologies. Additionally, the~integration of information theory, entropy, and~probability into our framework has enabled the construction of innovative probabilistic and entropic manifold--metric pairs, such as the GIPESTMMM manifold, which offer novel insights into infinite-dimensional and quotient spaces with applications in theoretical physics and~cosmology.

The methodologies employed, including the use of propositions to systematically verify higher-rank tensor constructions and the application of a partition of unity to derive smooth, positive-definite global metrics, have provided a robust foundation for these advancements. These results not only enhance the understanding of manifold--metric interactions but also open new avenues for modeling complex physical systems across scales, from~astronomical to cosmological phenomena. The~incorporation of symmetry and positive-definiteness conditions ensures that our metrics preserve essential geometric properties, making them suitable for applications in general relativity and beyond. Furthermore, the~probabilistic and entropic extensions of our formalism pave the way for interdisciplinary applications, bridging mathematics with information theory and statistical mechanics. Finally, this study can be extended through the use of advanced tensor theories~\cite{sym17050777}.

In conclusion, this work establishes a versatile and unified framework that pushes the boundaries of traditional manifold theory, offering both theoretical rigor and practical applicability. Future research could extend these findings by exploring additional codomain structures, such as octonionic metrics, or~by applying the formalism to specific physical systems, such as quantum field theories or emergent spacetime models. We anticipate that this framework will serve as a cornerstone for further investigations into the interplay of geometry, topology, and~physics, fostering new discoveries in both theoretical and applied~domains.


\vspace{6pt}
\funding{This research received no external~funding.}


\informedconsent{Not applicable.}

\dataavailability{No new data were created or analyzed in this study. Data sharing is {not applicable to this article.}} 

\acknowledgments{The author would like to acknowledge his former teachers, professors, and~colleagues for their constant questioning and improving of authors' knowledge. The~authors would like also to thank Herbert Spohn, Jeume Gomis, and Nick Early for their valuable discussions that improved the presentation of this study. {Part of this study} 
was accomplished during the COVID-19 pandemic;~therefore, the authors would like to express their sincere gratitude to all the social, medical, political staff, as~well as their friends, which made the pandemic less~painful.}

\conflictsofinterest{The author declares no conflict of~interest.} 

\begin{adjustwidth}{-\extralength}{0cm}

\reftitle{References}

\PublishersNote{}
\end{adjustwidth}
\end{document}